\documentclass[graybox]{svmult}

\usepackage{amsmath}
\usepackage{amssymb}
\usepackage{mathptmx}
\usepackage{helvet}
\usepackage{courier}
\usepackage{type1cm}

\usepackage{makeidx}

\usepackage{multicol}
\usepackage[bottom]{footmisc}
\usepackage{tikz}

\newcommand{\N}{\mathbb{N}}
\newcommand{\R}{\mathbb{R}}
\newcommand{\T}{\mathbb{T}}
\newcommand{\Z}{\mathbb{Z}}

\newcommand{\n}{\nabla}
\newcommand{\s}{\sigma}
\newcommand{\lra}{\rightarrow}
\newcommand{\lb}{\lbrace}
\newcommand{\rb}{\rbrace}

\makeindex

\begin{document}

\title*{Direct and Inverse Variational Problems on Time Scales: A Survey\thanks{This 
is a preprint of a paper whose final and definite form will be published in the 
Springer Volume \emph{Modeling, Dynamics, Optimization and Bioeconomics II},
Edited by A. A. Pinto and D. Zilberman (Eds.),
Springer Proceedings in Mathematics \& Statistics. Submitted 03/Sept/2014; 
Accepted, after a revision, 19/Jan/2016.}}

\author{Monika Dryl and Delfim F. M. Torres}

\institute{Monika Dryl
\at Center for Research and Development in Mathematics and Applications (CIDMA),\\
Department of Mathematics, University of Aveiro, 3810--193 Aveiro, Portugal\\
\email{monikadryl@ua.pt}
\and Delfim F. M. Torres
\at Center for Research and Development in Mathematics and Applications (CIDMA),\\
Department of Mathematics, University of Aveiro, 3810--193 Aveiro, Portugal\\
\email{delfim@ua.pt}}


\maketitle


\abstract{We deal with direct and inverse problems of the calculus of variations on arbitrary time scales.
Firstly, using the Euler--Lagrange equation and the strengthened Legendre condition, we give a general form
for a variational functional to attain a local minimum at a given point of the vector space.
Furthermore, we provide a necessary condition for a dynamic integro-differential equation to be an Euler--Lagrange equation
(Helmholtz's problem of the calculus of variations on time scales). New and interesting results for the
discrete and quantum settings are obtained as particular cases. Finally, we consider very general
problems of the calculus of variations given by the composition of a certain scalar function
with delta and nabla integrals of a vector valued field.}


\section{Introduction}

The theory of time scales is a relatively new area, which was introduced in 1988 by Stefan Hilger
in his Ph.D. thesis \cite{phdHilger,Hilger,Hilger2}. It bridges, generalizes and extends
the traditional discrete theory of dynamical systems (difference equations)
and the theory for continuous dynamical systems (differential equations) \cite{BohnerDEOTS}
and the various dialects of $q$-calculus \cite{Ernst,MyID:266} into a single unified theory
\cite{BohnerDEOTS,MBbook2003,Lakshmikantham}.

The calculus of variations on time scales was introduced in 2004 by Martin Bohner \cite{BohnerCOVOTS}
(see also \cite{MR1908827,HilscgerZeidan}) and has been developing rapidly
in the past ten years, mostly due to its great potential for applications, e.g.,
in biology \cite{BohnerDEOTS}, economics \cite{Almeida:iso,MR2218315,Atici:HMMS,MR2405376,natalia:CV}
and mathematics \cite{MR2662835,MR2668257,MR2879335,Torres}.
In order to deal with nontraditional applications in economics,
where the system dynamics are described on a time scale partly continuous and partly discrete,
or to accommodate nonuniform sampled systems, one needs to work with variational problems
defined on a time scale \cite{MR2218315,Atici:comparison,PhD:thesis:Monia,MyID:267}.

This survey is organized as follows. In Section~\ref{preliminaries}
we review the basic notions of the time-scale calculus:
the concepts of delta derivative and delta integral
(Section~\ref{sec:prelim:delta}); the analogous backward concepts
of nabla differentiation and nabla integration (Section~\ref{sec:prelim:nabla});
and the relation between delta/forward and nabla/backward approaches
(Section~\ref{subsec:relation:d:n}). Then, in Section~\ref{CoV},
we review the central results of the recent and powerful
calculus of variations on time scales. Both delta and nabla
approaches are considered (Sections~\ref{sec:calculus:of:variations}
and \ref{subsec:nabla:CoV}, respectively). Our results begin with
Section~\ref{chapter_5_P3+P4}, where we investigate inverse problems
of the calculus of variations on time scales. To our best knowledge,
and in contrast with the direct problem, which is already well studied
in the framework of time scales \cite{Torres},
the inverse problem has not been studied before. Its investigation
is part of the Ph.D. thesis of the first author \cite{PhD:thesis:Monia}.
Let here, for the moment and just for simplicity, the time scale $\mathbb{T}$
be the set $\mathbb{R}$ of real numbers. Given $L$, a Lagrangian function,
in the ordinary/direct fundamental problem of the calculus of variations
one wants to find extremal curves $y : [a, b] \rightarrow \mathbb{R}^n$
giving stationary values to some action integral (functional)
$$
\mathcal{I}(y)=\int\limits_{a}^{b} L(t,y(t),y'(t)) dt
$$
with respect to variations of $y$ with fixed boundary conditions $y(a)=y_{a}$ and $y(b)=y_{b}$.
Thus, if in the direct problem we start with a Lagrangian
and we end up with extremal curves, then one might expect as inverse problem
to start with extremal curves and search for a Lagrangian. Such inverse
problem is considered, in the general context of time scales, in Section~\ref{P3}:
we describe a general form of a variational functional having an extremum at a given function
$y_{0}$ under Euler--Lagrange and strengthened Legendre conditions (Theorem~\ref{P3:th:integrand:2}).
In Corollary~\ref{P3:cor:isolated} the form of the Lagrangian $L$
on the particular case of an isolated time scale is presented and we end
Section~\ref{P3} with some concrete cases and examples.
We proceed with a more common inverse problem of the calculus of variations
in Section~\ref{P4}. Indeed, normally the starting point are not the extremal curves but,
instead, the Euler--Lagrange equations that such curves must satisfy:
\begin{equation}
\label{eq:EL}
\frac{\partial L}{\partial y} - \frac{d}{dt} \frac{\partial L}{\partial y'} = 0
\Leftrightarrow
\frac{\partial L}{\partial y}
- \frac{\partial^2 L}{\partial t \partial y'}
- \frac{\partial^2 L}{\partial y \partial y'} y'
- \frac{\partial^2 L}{\partial y' \partial y'} y'' = 0
\end{equation}
(we are still keeping, for illustrative purposes, $\mathbb{T} = \mathbb{R}$).
This is what is usually known as the inverse problem of the calculus of variations:
start with a second order ordinary differential equation and determine a Lagrangian $L$
(if it exists) whose Euler--Lagrange equations \emph{are the same as}
the given equation. The problem of variational formulation of differential equations
(or the inverse problem of the calculus of variations) dates
back to the 19th century. The problem seems to have been posed by Helmholtz in 1887,
followed by several results from Darboux (1894), Mayer (1896), Hirsch (1897),
Volterra (1913) and Davis (1928, 1929) \cite{MR2732180}.
There are, however, two different types of inverse problems,
depending on the meaning of the phrase ``are the same as''.
Do we require the equations to be the same or do we allow multiplication by functions
to obtain new but equivalent equations?
The first case is often called \emph{Helmholtz's inverse problem}:
find conditions under which a given differential equation is an Euler--Lagrange equation.
The latter case is often called the \emph{multiplier problem}:
given $f(t,y,y',y'') = 0$, does a function $r(t,y,y')$ exist such that the equation
$r(t,y,y') f(t,y,y',y'') = 0$ is the Euler--Lagrange equation of a functional?
In this work we are interested in Helmholtz's problem. The answer to this problem
in $\T = \R$ is classical and well known: the complete solution to Helmholtz's problem
is found in the celebrated 1941 paper of Douglas \cite{Douglas}.
Let $O$ be a second order differential operator.
Then, the differential equation $O(y) = 0$ is a second order Euler--Lagrange equation
if and only if the Fr\'{e}chet derivatives of $O$ are self-adjoint.
A simple example illustrating the difference between both inverse problems
is the following one. Consider the second order differential equation
\begin{equation}
\label{eq:EL:ex:introd:fric}
m y'' + h y' + k y = f.
\end{equation}
This equation is not self-adjoint and, as a consequence,
there is no variational problem with such Euler--Lagrange equation.
However, if we multiply  the equation by $p(t) = \exp(ht/m)$, then
\begin{equation}
\label{eq:EL:ex:introd}
m \frac{d}{dt} \left[\exp(ht/m) y'\right] + k \, \exp(ht/m) y = \exp(ht/m) f
\end{equation}
and now a variational formulation is possible: the Euler--Lagrange equation
\eqref{eq:EL} associated with
$$
\mathcal{I}(y)=\int\limits_{t_{0}}^{t_{1}} \exp(ht/m)
\left[ \frac{1}{2} m y'^2 - \frac{1}{2} k y^2 - f y \right] dt
$$
is precisely \eqref{eq:EL:ex:introd}. A recent theory of the calculus
of variations that allows to obtain \eqref{eq:EL:ex:introd:fric}
directly has been developed, but involves Lagrangians depending
on fractional (noninteger) order derivatives \cite{livro:FC:Comput,livro:Adv:CoV,livro:FC:Int}.
For a survey on the fractional calculus of variations, which is not our subject here,
we refer the reader to \cite{MyID:288}. For the time scale $\T = \Z$,
available results on the inverse problem of the calculus of variations
are more recent and scarcer. In this case Helmholtz's inverse problem can be formulated as follows:
find conditions under which a second order difference equation
is a second order discrete Euler--Lagrange equation.
Available results in the literature go back to the works of
Cr\u{a}ciun and Opri\c{s} (1996) and Albu and Opri\c{s} (1999)
and, more recently, to the works of Hydon and Mansfeld (2004)
and Bourdin and Cresson (2013) \cite{MR1770981,Helmholtz,MR1609049,MR2049870}.
The main difficulty to obtain analogous results to those
of the classical continuous calculus of variations
in the discrete (or, more generally, in the time-scale) setting
is due to the lack of chain rule. This lack of chain rule
is easily seen with a simple example.
Let $f,g: \mathbb{Z} \rightarrow \mathbb{Z}$ be defined by
$f\left(  t\right)  =t^{3}$, $g\left(  t\right)  =2t$.
Then,
$\Delta\left(  f\circ g\right) = \Delta\left(  8t^{3}\right)=8\left(
3t^{2}+3t+1\right)  =24t^{2}+24t+8$
and
$\Delta f\left(  g\left(  t\right)  \right)  \cdot \Delta g\left(
t\right)  =\left(  12t^{2}+6t+1\right)  2=24t^{2}+12t+2$.
Therefore,
$\Delta \left(  f\circ g\right) \neq \Delta f\left(  g\left(  t\right)
\right)  \Delta g\left(  t\right)$.
The difficulties caused by the lack of a chain rule in a general time scale $\T$,
in the context of the inverse problem of the calculus of variations on time scales,
are discussed in Section~\ref{final remarks}. To deal with the problem,
our approach to the inverse problem of the calculus
of variations uses an integral perspective instead of the classical differential point of view.
As a result, we obtain a useful tool to identify integro-differential
equations which are not Euler--Lagrange equations on an arbitrary time scale $\mathbb{T}$.
More precisely, we define the notion of self-adjointness of a first order integro-differential equation
(Definition~\ref{P4:def:self:adj:int:diff}) and its equation of variation (Definition~\ref{P4:def:eq:var}).
Using such property, we prove a necessary condition for an integro-differential equation
on an arbitrary time scale $\mathbb{T}$ to be an Euler--Lagrange equation (Theorem~\ref{P4:th:necess:EL:int:diff}).
In order to illustrate our results we present Theorem~\ref{P4:th:necess:EL:int:diff} in the particular time scales
$\mathbb{T}\in\left\{ \mathbb{R},h\mathbb{Z},\overline{q^{\mathbb{Z}}}\right\}$, $h > 0$, $q>1$
(Corollaries~\ref{P4:cor:R}, \ref{P4:cor:hZ}, and \ref{P4:cor:qZ}).
Furthermore, we discuss equivalences between: (i) integro-differential equations
\eqref{P4:eq:int:diff:1} and second order differential equations \eqref{P4:eq:6}
(Proposition~\ref{P4:prop:1}), and (ii) equations of variations of them
on an arbitrary time scale $\mathbb{T}$ (\eqref{P4:eq:eq:var}
and \eqref{P4:eq:eq:var:1}, respectively). As a result, we show
that it is impossible to prove the latter equivalence due to lack of
a general chain rule on an arbitrary time scale \cite{BohGus1,BohnerDEOTS}.
In Section~\ref{sec:mr} we address the direct problem of the calculus
of variations on time scales by considering
a variational problem which may be found often
in economics (see \cite{MalinowskaTorresCompositionDelta} and references therein).
We extremize a functional of the calculus of variations that is the composition
of a certain scalar function with the delta and nabla integrals
of a vector valued field, possibly subject to boundary conditions and/or isoperimetric constraints.
In Section~\ref{subsec:EL} we provide general Euler--Lagrange equations in integral form (Theorem~\ref{P6:th:main}),
transversality conditions are given in Section~\ref{sec:nbc}, while Section~\ref{sub:sec:iso:p}
considers necessary optimality conditions for isoperimetric problems on an arbitrary time scale.
Interesting corollaries and examples are presented in Section~\ref{sec:examples}.
We end with Section~\ref{sec:conc} of conclusions and open problems.


\section{Preliminaries}
\label{preliminaries}

A time scale $\mathbb{T}$
is an arbitrary nonempty closed subset of $\mathbb{R}$.
The set of real numbers $\R$, the integers $\Z$, the natural numbers $\N$, the nonnegative integers
$\N_{0}$, an union of closed intervals $[0,1]\cup [2,7]$
or the Cantor set are examples of time scales,
while the set of rational numbers $\mathbb{Q}$, the irrational numbers
$\R \setminus\mathbb{Q}$, the complex numbers $\mathbb{C}$
or an open interval like $(0,1)$ are not time scales.
Throughout this survey we assume that for $a,b\in\mathbb{T}$,
$a<b$, all intervals are time scales intervals, i.e.,
$[a,b]=[a,b]_{\mathbb{T}}:=[a,b]\cap\mathbb{T}
=\left\{t\in\mathbb{T}: a\leq t\leq b\right\}$.

\begin{definition}[See Section~1.1 of \cite{BohnerDEOTS}]
\label{def:jump:op}
Let $\T$ be a time scale and $t\in\T$.
The forward jump operator $\sigma:\mathbb{T} \rightarrow \mathbb{T}$ is defined by
$\sigma(t):=\inf\left\{ s\in\mathbb{T}: s>t\right\}$ for $t\neq \sup\mathbb{T}$
and $\sigma(\sup\mathbb{T}) := \sup\mathbb{T}$ if $\sup\mathbb{T}<+\infty$.
Accordingly, we define the backward jump operator $\rho:\mathbb{T} \rightarrow \mathbb{T}$
by $\rho(t):=\sup\left\{ s\in\mathbb{T}: s<t\right\}$ for
$t\neq \inf\mathbb{T}$ and $\rho(\inf\mathbb{T}) := \inf\mathbb{T}$ if $\inf\mathbb{T}>-\infty$.
The forward graininess function
$\mu:\mathbb{T} \rightarrow [0,\infty)$ is defined by $\mu(t):=\sigma(t)-t$ and
the backward graininess function $\nu:\mathbb{T} \rightarrow [0,\infty)$
by $\nu(t):=t-\rho(t)$.
\end{definition}

\begin{example}
The two classical time scales are $\mathbb{R}$ and $\mathbb{Z}$,
representing the continuous and the purely discrete time, respectively.
Other standard examples are the periodic numbers,
$h\mathbb{Z}=\left\{ hk: h>0, k\in\Z\right\}$,
and the $q$-scale
$$
\overline{q^{\mathbb{Z}}}
:=q^{\Z}\cup\left\{ 0\right\}=\left\{ q^{k}:q>1, k\in\mathbb{Z}\right\}\cup\left\{ 0\right\}.
$$
Sometimes one considers also the time scale $q^{\mathbb{N}_{0}}=\left\{ q^{k}:q>1, k\in\mathbb{N}_{0}\right\}$.
The following time scale is common:
$\mathbb{P}_{a,b}=\bigcup\limits_{k=0}^{\infty} [k(a+b),k(a+b)+a]$, $a,b>0$.
\end{example}

Table~\ref{tbl:1} and Example~\ref{ex:1} present different forms of jump operators
$\s$ and $\rho$, and graininiess functions $\mu$ and $\nu$, in specified time scales.
\begin{table}[ht]
\begin{center}
\begin{tabular}{|c|c|c|c|}
\hline
$\T$  & $\R$ & $h\Z$ & $\overline{q^{\Z}}$ \\ \hline
$\sigma(t)$ & $t$ & $t+h$ &$qt$   \\ \hline
$\rho(t)$  & $t$ & $t-h$ & $\frac{t}{q}$  \\ \hline
$\mu(t)$ & $0$& $h$ & $t(q-1)$   \\ \hline
$\nu(t)$  & $0$ & $h$ & $\frac{t(q-1)}{q}$ \\ \hline
\end{tabular}
\end{center}
\caption{\label{tbl:1}Examples of jump operators and graininess functions on different time scales.}
\end{table}

\begin{example}[See Example~1.2 of \cite{Wyrwas}]
\label{ex:1}
Let $a,b>0$ and consider the time scale
$$
\mathbb{P}_{a,b}=\bigcup\limits_{k=0}^{\infty} [k(a+b),k(a+b)+a].
$$
Then,
$$
\sigma(t)=
\begin{cases}
t &\text{if } t\in A_{1},
\\
t+b &\text{if } t\in A_{2},
\end{cases}
\quad\quad\quad
\rho(t)=
\begin{cases}
t-b &\text{if } t\in B_{1},
\\
t &\text{if } t\in B_{2}
\end{cases}
$$
\begin{figure}
\centering
\unitlength 1mm
\linethickness{0.2pt}
\ifx\plotpoint\undefined\newsavebox{\plotpoint}\fi
\qquad \qquad \qquad \qquad
\begin{picture}(13,35)(0,-13)
\put(-15,-15){\vector(0,2){30}}
\linethickness{1pt}
\put(-15,-15){\line(1,0){6}}
\put(-6,-15){\line(1,0){6}}
\put(3,-15){\line(1,0){6}}
\put(12,-15){\line(1,0){6}}
\put(16,-15){\vector(1,0){2}}
\put(-15,-17){\makebox(0,0)[cc]{\scriptsize $0$}}
\put(-9,-17){\makebox(0,0)[cc]{\scriptsize $a$}}
\put(17,-17){\makebox(0,0)[cc]{\scriptsize $t$}}
\put(-15.5,-9){\line(1,0){1}}
\put(-15.5,-6){\line(1,0){1}}
\put(-15.5,0){\line(1,0){1}}
\put(-15.5,3){\line(1,0){1}}
\put(-15.5,9){\line(1,0){1}}
\put(-15.5,12){\line(1,0){1}}
\put(-17,-9){\makebox(0,0)[rc]{\scriptsize $a$}}
\put(-17,-6){\makebox(0,0)[rc]{\scriptsize $a+b$}}
\put(-17,0){\makebox(0,0)[rc]{\scriptsize $2a+b$}}
\thicklines
\put(-15,-15){\line(1,1){5.55}}
\put(-6,-6){\line(1,1){5.55}}
\put(3,3){\line(1,1){5.55}}
\put(12,12){\line(1,1){4}}
\put(-9,-9){\circle{1}}
\put(-9,-6){\circle*{1}}
\put(-6,-6){\circle*{1}}
\put(0,0){\circle{1}}
\put(0,3){\circle*{1}}
\put(3,3){\circle*{1}}
\put(9,9){\circle{1}}
\put(9,12){\circle*{1}}
\put(12,12){\circle*{1}}
\put(12,3){\makebox(0,0)[cc]{$\sigma(t)$}}
\end{picture}
\qquad \qquad \qquad \qquad
\qquad \qquad \qquad \qquad
\begin{picture}(13,25)(0,-13)
\put(-15,-15){\vector(0,2){30}}
\linethickness{0.5pt}
\put(-15,-15){\line(1,0){6}}
\put(-6,-15){\line(1,0){6}}
\put(3,-15){\line(1,0){6}}
\put(12,-15){\line(1,0){6}}
\put(16,-15){\vector(1,0){2}}
\put(-15,-17){\makebox(0,0)[cc]{\scriptsize $0$}}
\put(-9,-17){\makebox(0,0)[cc]{\scriptsize $a$}}
\put(17,-17){\makebox(0,0)[cc]{\scriptsize $t$}}
\put(-15.5,-9){\line(1,0){1}}
\put(-15.5,-6){\line(1,0){1}}
\put(-15.5,0){\line(1,0){1}}
\put(-15.5,3){\line(1,0){1}}
\put(-15.5,9){\line(1,0){1}}
\put(-15.5,12){\line(1,0){1}}
\put(-17,-9){\makebox(0,0)[rc]{\scriptsize $a$}}
\put(-17,-6){\makebox(0,0)[rc]{\scriptsize $a+b$}}
\put(-17,0){\makebox(0,0)[rc]{\scriptsize $2a+b$}}
\thicklines
\put(-15,-15){\line(1,1){5.6}}
\put(-5.55,-5.55){\line(1,1){5.6}}
\put(3.45,3.45){\line(1,1){5.6}}
\put(12.45,12.45){\line(1,1){4}}
\put(-9,-9){\circle*{1}}
\put(-6,-9){\circle*{1}}
\put(-6,-6){\circle{1}}
\put(0,0){\circle*{1}}
\put(3,0){\circle*{1}}
\put(3,3){\circle{1}}
\put(9,9){\circle*{1}}
\put(12,9){\circle*{1}}
\put(12,12){\circle{1}}
\put(12,3){\makebox(0,0)[cc]{$\rho(t)$}}
\end{picture}
\newline
\caption{Jump operators of the time scale
$\mathbb{T}=\mathbb{P}_{a,b}$.\label{Fig1}}
\end{figure}
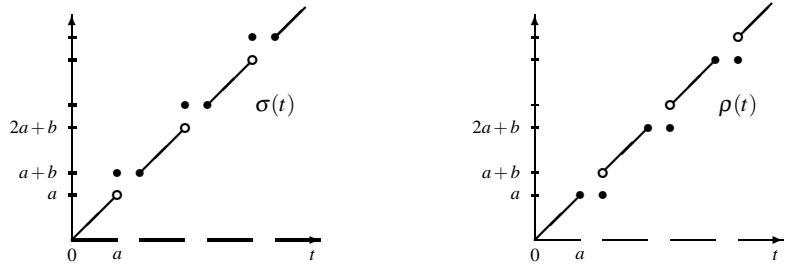
(see Figure~\ref{Fig1}) and
$$
\mu(t)=
\begin{cases}
0 &\text{if } t\in A_{1},
\\
b &\text{if } t\in A_{2},
\end{cases}
\quad\quad\quad
\nu(t)=
\begin{cases}
b &\text{if } t\in B_{1},
\\
0 &\text{if } t\in B_{2},
\end{cases}
$$
where
$$
\bigcup\limits_{k=0}^{\infty} [k(a+b),k(a+b)+a]=A_{1}\cup A_{2}=B_{1}\cup B_{2}
$$
with
\begin{equation*}
A_{1}=\bigcup\limits_{k=0}^{\infty} [k(a+b),k(a+b)+a),
\quad
B_{1}=\bigcup\limits_{k=0}^{\infty} \left\{ k(a+b)\right\},
\end{equation*}
\begin{equation*}
A_{2}=\bigcup\limits_{k=0}^{\infty} \left\{ k(a+b)+a\right\},
\quad
B_{2}=\bigcup\limits_{k=0}^{\infty} (k(a+b),k(a+b)+a].
\end{equation*}
\end{example}

In the time-scale theory the following classification of points is used:
\begin{itemize}
\item
A point $t\in\mathbb{T}$ is called
\emph{right-scattered}
or \emph{left-scattered}
if $\sigma(t)>t$ or $\rho(t)<t$, respectively.
\item
A point $t$ is \emph{isolated} if $\rho(t)<t<\sigma(t)$.
\item
If $t<\sup\T$ and $\sigma(t)=t$, then $t$ is called
\emph{right-dense};
if $t>\inf \T$ and $\rho(t)=t$, then $t$ is called
\emph{left-dense}.
\item We say that $t$ is \emph{dense} if $\rho(t)=t=\sigma(t)$.
\end{itemize}

\begin{definition}[See Section~1 of \cite{MR2028477}]
A time scale $\mathbb{T}$ is said to be an isolated time scale provided given any
$t \in \mathbb{T}$, there is a $\delta > 0$ such that
$(t - \delta, t+\delta) \cap \mathbb{T} = \{t\}$.
\end{definition}

\begin{definition}[See \cite{Bartos1}]
\label{def:regular}
A time scale $\mathbb{T}$ is said to be regular
if the following two conditions
are satisfied simultaneously for all $t\in\mathbb{T}$:
$\sigma(\rho(t))=t$ and $\rho(\sigma(t))=t$.
\end{definition}


\subsection{The delta derivative and the delta integral}
\label{sec:prelim:delta}

If $f:\T\rightarrow \R$, then we define $f^{\sigma}:\T\rightarrow\R$
by $f^{\s}(t):=f(\s(t))$ for all $t\in\T$. The delta derivative
(or \emph{Hilger derivative}) of function $f:\T\lra\R$
is defined for points in the set $\T^{\kappa}$, where
\begin{equation*}
\T^{\kappa} :=
\begin{cases}
\T\setminus\left\{\sup\T\right\}
& \text{ if } \rho(\sup\T)<\sup\T<\infty,\\
\T
& \text{ otherwise. }
\end{cases}
\end{equation*}

Let us define the sets $\T^{\kappa^n}$,  $n\geq 2$, inductively:
$\T^{\kappa^1} :=\T^\kappa$ and
$\T^{\kappa^n} := \left(\T^{\kappa^{n-1}}\right)^\kappa$,
$n\geq 2$. We define delta differentiability in the following way.

\begin{definition}[Section 1.1 of \cite{BohnerDEOTS}]
\label{def:differ:delta}
Let $f:\mathbb{T}\rightarrow\mathbb{R}$ and $t\in\mathbb{T}^{\kappa}$.
We define $f^{\Delta}(t)$ to be the number (provided it exists)
with the property that given any $\varepsilon >0$,
there is a neighborhood $U$ ($U=(t-\delta, t+\delta)\cap\T$ for some $\delta>0$) of $t$ such that
\begin{equation*}
\left|f^{\sigma}(t)-f(s)-f^{\Delta}(t)\left(\sigma(t)-s\right)\right|
\leq \varepsilon \left|\sigma(t)-s\right| \mbox{ for all }  s\in U.
\end{equation*}
A function $f$ is delta differentiable on $\mathbb{T}^{\kappa}$ provided
$f^{\Delta}(t)$ exists for all $t\in\mathbb{T}^{\kappa}$. Then,
$f^{\Delta}:\mathbb{T}^{\kappa}\rightarrow\mathbb{R}$
is called the delta derivative
of $f$ on $\mathbb{T}^{\kappa}$.
\end{definition}

\begin{theorem}[Theorem~1.16 of \cite{BohnerDEOTS}]
\label{th:differ:delta}
Let $f:\mathbb{T} \rightarrow \mathbb{R}$
and $t\in\mathbb{T}^{\kappa}$. The following hold:
\begin{enumerate}
\item
If $f$ is delta differentiable at $t$, then $f$ is continuous at $t$.
\item
If $f$ is continuous at $t$ and $t$ is right-scattered,
then $f$ is delta differentiable at $t$ with
$$
f^{\Delta}(t)=\frac{f^\sigma(t)-f(t)}{\mu(t)}.
$$
\item
If $t$ is right-dense, then $f$ is delta differentiable at $t$
if and only if the limit
$$
\lim\limits_{s\rightarrow t}\frac{f(t)-f(s)}{t-s}
$$
exists as a finite number. In this case,
$$
f^{\Delta}(t)=\lim\limits_{s\rightarrow t}\frac{f(t)-f(s)}{t-s}.
$$
\item
If $f$ is delta differentiable at $t$, then
$$f^\sigma(t)=f(t)+\mu(t)f^{\Delta}(t).$$
\end{enumerate}
\end{theorem}

The next example is a consequence of Theorem~\ref{th:differ:delta}
and presents different forms of the delta derivative on specific time scales.

\begin{example}
Let $\T$ be a time scale.
\begin{enumerate}
\item
If $\mathbb{T}=\mathbb{R}$, then $f:\mathbb{R} \rightarrow \mathbb{R}$
is delta differentiable at $t\in\mathbb{R}$ if and only if
\begin{equation*}
f^\Delta(t)=\lim\limits_{s\rightarrow t}\frac{f(t)-f(s)}{t-s}
\end{equation*}
exists, i.e., if and only if $f$ is differentiable
(in the ordinary sense) at $t$ and in this case we have
$f^{\Delta}(t)=f'(t)$.
\item
If $\mathbb{T}=h\mathbb{Z}$, $h > 0$, then
$f:h\mathbb{Z} \rightarrow \mathbb{R}$ is delta differentiable
at $t\in h\mathbb{Z}$ with
\begin{equation*}
f^{\Delta}(t)=\frac{f(\sigma(t))-f(t)}{\mu(t)}=\frac{f(t+h)-f(t)}{h}=:\Delta_{h}f(t).
\end{equation*}
In the particular case $h=1$ we have $f^{\Delta}(t)=\Delta f(t)$,
where $\Delta$ is the usual forward difference operator.
\item
If $\mathbb{T}=\overline{q^{\mathbb{Z}}}$, $q>1$,
then for a delta differentiable function $f:\overline{q^{\mathbb{Z}}}\lra\R$ we have
\begin{equation*}
f^{\Delta}(t)=\frac{f(\sigma(t))-f(t)}{\mu(t)}=\frac{f(qt)-f(t)}{(q-1)t}=:\Delta_{q}f(t)
\end{equation*}
for all $t\in\overline{q^{\mathbb{Z}}}\setminus \left\{ 0\right\}$,
i.e., we get the usual Jackson derivative of quantum calculus \cite{QC,MyID:266}.
\end{enumerate}
\end{example}

Now we formulate some basic properties of the delta derivative on time scales.

\begin{theorem}[Theorem~1.20 of \cite{BohnerDEOTS}]
\label{th:differ:prop:delta}
Let $f,g:\mathbb{T} \rightarrow \mathbb{R}$
be delta differentiable at $t\in\mathbb{T^{\kappa}}$. Then,
\begin{enumerate}
\item the sum $f+g:\mathbb{T} \rightarrow \mathbb{R}$ is delta differentiable at $t$ with
\begin{equation*}
(f+g)^{\Delta}(t)=f^{\Delta}(t)+g^{\Delta}(t);
\end{equation*}
\item for any constant $\alpha$, $\alpha f:\mathbb{T}\rightarrow\mathbb{R}$
is delta differentiable at $t$ with
\begin{equation*}
(\alpha f)^{\Delta}(t)=\alpha f^{\Delta}(t);
\end{equation*}
\item the product $fg:\mathbb{T}\rightarrow\mathbb{R}$
is delta differentiable at $t$ with
\begin{equation*}
(fg)^{\Delta}(t)=f^{\Delta}(t)g(t)+f^{\sigma}(t)g^{\Delta}(t)
=f(t)g^{\Delta}(t)+f^{\Delta}(t)g^{\sigma}(t);
\end{equation*}
\item if $g(t)g^{\sigma}(t)\neq 0$, then $f/g$ is delta differentiable at $t$ with
\begin{equation*}
\left(\frac{f}{g}\right)^{\Delta}(t)
=\frac{f^{\Delta}(t)g(t)-f(t)g^{\Delta}(t)}{g(t)g^{\sigma}(t)}.
\end{equation*}
\end{enumerate}
\end{theorem}

Now we introduce the theory of delta integration on time scales.
We start by defining the associated class of functions.

\begin{definition}[Section~1.4 of \cite{MBbook2003}]
A function $f:\mathbb{T}\rightarrow \mathbb{R}$ is called rd-continuous
provided it is continuous at right-dense points in $\mathbb{T}$ and its
left-sided limits exist (finite) at all left-dense points in $\mathbb{T}$.
\end{definition}
The set of all rd-continuous functions
$f:\mathbb{T} \rightarrow \mathbb{R}$ is denoted by
$C_{rd}=C_{rd}(\mathbb{T})=C_{rd}(\mathbb{T},\mathbb{R})$.
The set of functions $f:\mathbb{T}\rightarrow \mathbb{R}$
that are delta differentiable and whose derivative
is rd-continuous is denoted by
$C^{1}_{rd}=C_{rd}^{1}(\mathbb{T})=C^{1}_{rd}(\mathbb{T},\mathbb{R})$.

\begin{definition}[Definition~1.71 of \cite{BohnerDEOTS}]
A function $F:\mathbb{T} \rightarrow \mathbb{R}$ is called
a delta antiderivative
of $f:\mathbb{T} \rightarrow \mathbb{R}$ provided
$F^{\Delta}(t)=f(t)$ for all $t\in\mathbb{T}^{\kappa}$.
\end{definition}

\begin{definition}
Let $\mathbb{T}$ be a time scale and $a,b\in\mathbb{T}$.
If $f:\mathbb{T}^{\kappa} \rightarrow \mathbb{R}$ is a rd-continuous
function and $F:\mathbb{T} \rightarrow \mathbb{R}$
is an antiderivative of $f$, then the Cauchy delta integral
is defined by
\begin{equation*}
\int\limits_{a}^{b} f(t)\Delta t := F(b)-F(a).
\end{equation*}
\end{definition}

\begin{theorem}[Theorem~1.74 of \cite{BohnerDEOTS}]
Every rd-continuous function $f$ has an antiderivative $F$.
In particular, if $t_{0}\in\mathbb{T}$, then $F$ defined by
\begin{equation*}
F(t):=\int\limits_{t_{0}}^{t} f(\tau)\Delta \tau, \quad t\in\mathbb{T},
\end{equation*}
is an antiderivative of $f$.
\end{theorem}

\begin{theorem}[Theorem~1.75 of \cite{BohnerDEOTS}]
\label{th:int:lim:sigma}
If $f\in C_{rd}$, then
$\displaystyle \int\limits_{t}^{\sigma(t)}f(\tau)\Delta \tau=\mu(t)f(t)$,
$t\in \mathbb{T}^{\kappa}$.
\end{theorem}

Let us see two special cases of the delta integral.

\begin{example}
\label{ex:int:R:hZ:qN}
Let $a,b\in\mathbb{T}$ and $f:\mathbb{T} \rightarrow \mathbb{R}$ be rd-continuous.
\begin{enumerate}
\item
If $\mathbb{T}=\mathbb{R}$, then
\begin{equation*}
\int\limits_{a}^{b}f(t)\Delta t=\int\limits_{a}^{b}f(t)dt,
\end{equation*}
where the integral on the right hand side is the usual Riemann integral.
\item
If $[a,b]$ consists of only isolated points, then
\begin{equation*}
\int\limits_{a}^{b} f(t)\Delta t=
\begin{cases}
\sum\limits_{t\in[a,b)}\mu(t)f(t), & \hbox{ if } a<b, \\
0, & \hbox{ if } a=b,\\
-\sum\limits_{t\in[b,a)}\mu(t)f(t), & \hbox{ if } a>b.
\end{cases}
\end{equation*}
\end{enumerate}
\end{example}

Now we present some useful properties of the delta integral.

\begin{theorem}[Theorem~1.77 of \cite{BohnerDEOTS}]
\label{th:int:prop:delta}
If $a,b,c\in\mathbb{T}$, $a < c < b$,
$\alpha\in\mathbb{R}$, and $f,g \in C_{rd}(\mathbb{T}, \mathbb{R})$, then:
\begin{enumerate}
\item
$\int\limits_{a}^{b}(f(t)+g(t))\Delta t
=\int\limits_{a}^{b} f(t)\Delta t+\int\limits_{a}^{b}g(t)\Delta t$,
\item
$\int\limits_{a}^{b}\alpha f(t)\Delta t=\alpha \int\limits_{a}^{b} f(t)\Delta t$,
\item
$\int\limits_{a}^{b}f(t)\Delta t
=-\int\limits_{b}^{a}f(t)\Delta t$,
\item
$\int\limits_{a}^{b} f(t)\Delta t
=\int\limits_{a}^{c} f(t)\Delta t
+\int\limits_{c}^{b} f(t)\Delta t$,
\item
$\int\limits_{a}^{a} f(t)\Delta t=0$,
\item
$\int\limits_{a}^{b}f(t)g^{\Delta}(t)\Delta t
=\left.f(t)g(t)\right|^{t=b}_{t=a}
-\int\limits_{a}^{b} f^{\Delta}(t)g^\sigma(t)\Delta t$,
\item
$\int\limits_{a}^{b} f^\sigma(t) g^{\Delta}(t)\Delta t
=\left.f(t)g(t)\right|^{t=b}_{t=a}
-\int\limits_{a}^{b}f^{\Delta}(t)g(t)\Delta t$.
\end{enumerate}
\end{theorem}


\subsection{The nabla derivative and the nabla integral}
\label{sec:prelim:nabla}

The nabla calculus is similar to the delta one of Section~\ref{sec:prelim:delta}.
The difference is that the backward jump operator $\rho$
takes the role of the forward jump operator $\s$.
For a function $f:\T\lra\R$ we define $f^{\rho}:\T\lra\R$ by $f^{\rho}(t):=f(\rho(t))$.
If $\mathbb{T}$ has a right-scattered minimum $m$, then
we define $\mathbb{T}_{\kappa}:=\mathbb{T}-\left\{ m\right\}$;
otherwise, we set $\mathbb{T}_{\kappa}:=\mathbb{T}$:
\begin{equation*}
\T_{\kappa} :=
\begin{cases}
\T\setminus\left\{\inf\T\right\}
& \text{ if } -\infty<\inf\T<\s(\inf\T),\\
\T
& \text{ otherwise}.
\end{cases}
\end{equation*}

Let us define the sets $\T_{\kappa}$, $n\geq 2$, inductively: $\mathbb{T}_{\kappa^1} := \mathbb{T}_\kappa$ and
$\mathbb{T}_{\kappa^n} := (\mathbb{T}_{\kappa^{n-1}})_\kappa$, $n\geq 2$.
Finally, we define $\mathbb{T}_{\kappa}^{\kappa} := \mathbb{T}_{\kappa} \cap \mathbb{T}^{\kappa}$.
The definition of nabla derivative of a function $f:\T\lra\R$ at point $t\in\T_{\kappa}$ is similar
to the delta case (cf. Definition~\ref{def:differ:delta}).

\begin{definition}[Section 3.1 of \cite{MBbook2003}]
We say that a function $f:\mathbb{T} \rightarrow \mathbb{R}$ is nabla differentiable
at $t\in\mathbb{T}_{\kappa}$ if there is a number $f^{\nabla}(t)$ such that for all
$\varepsilon >0$ there exists a neighborhood $U$ of $t$
(i.e., $U=(t-\delta, t+\delta)\cap\mathbb{T}$ for some $\delta>0$) such that
\begin{equation*}
|f^{\rho}(t)-f(s)-f^{\nabla}(t)(\rho(t)-s)| \leq
\varepsilon |\rho(t)-s| \, \mbox{ for all } s\in U.
\end{equation*}
We say that $f^{\nabla}(t)$ is the nabla derivative of $f$ at $t$.
Moreover, $f$ is said to be nabla differentiable on $\mathbb{T}$ provided $f^{\nabla}(t)$
exists for all $t\in\mathbb{T}_{\kappa}$.
\end{definition}

The main properties of the nabla derivative
are similar to those given in Theorems~\ref{th:differ:delta}
and \ref{th:differ:prop:delta}, and can be found, respectively,
in Theorems~8.39 and 8.41 of \cite{BohnerDEOTS}.

\begin{example}
If $\T=\R$, then $f^{\n}(t)=f'(t)$.
If $\T=h\Z$, $h>0$, then
$$
f^{\n}(t)=\frac{f(t)-f(t-h)}{h} =: \n_{h} f(t).
$$
For $h=1$ the operator $\n_{h}$ reduces to the standard backward difference operator
$\nabla f(t) = f(t) - f(t-1)$.
\end{example}

We now briefly recall the theory of nabla integration on time scales.
Similarly as in the delta case, first we define a suitable class of functions.

\begin{definition}[Section~3.1 of \cite{MBbook2003}]
Let $\mathbb{T}$ be a time scale and $f:\mathbb{T}\rightarrow \mathbb{R}$.
We say that $f$ is ld-continuous if it is continuous at left-dense
points $t\in\T$ and its right-sided limits exist (finite) at all right-dense points.
\end{definition}

\begin{remark}
If $\T=\R$, then $f$ is ld-continuous if and only if $f$ is continuous.
If $\T=\Z$, then any function is ld-continuous.
\end{remark}

The set of all ld-continuous functions $f:\mathbb{T}\rightarrow \mathbb{R}$
is denoted by $C_{ld}=C_{ld}(\mathbb{T})=C_{ld}(\mathbb{T},\mathbb{R})$;
the set of all nabla differentiable functions with ld-continuous derivative by
$C^{1}_{ld}=C^{1}_{ld}(\mathbb{T})=C^{1}_{ld}(\mathbb{T},\mathbb{R})$.
Follows the definition of nabla integral on time scales.

\begin{definition}[Definition~8.42 of \cite{BohnerDEOTS}]
A function $F : \mathbb{T} \rightarrow \mathbb{R}$
is called a nabla antiderivative
of $f : \mathbb{T} \rightarrow \mathbb{R}$
provided $F^\nabla(t) = f(t)$ for all $t \in \mathbb{T}_\kappa$.
In this case we define the nabla integral of $f$ from $a$ to $b$
($a, b \in \mathbb{T}$) by
$$
\int_a^b f(t) \nabla t := F(b) - F(a).
$$
\end{definition}

\begin{theorem}[Theorem~8.45 of \cite{BohnerDEOTS}
or Theorem~11 of \cite{TorresDeltaNabla}]
Every ld-continuous function $f$ has a nabla antiderivative $F$.
In particular, if $a \in \mathbb{T}$, then $F$ defined by
$$
F(t) = \int_a^t f(\tau) \nabla \tau, \quad t \in \mathbb{T},
$$
is a nabla antiderivative of $f$.
\end{theorem}

\begin{theorem}[Theorem~8.46 of \cite{BohnerDEOTS}]
\label{th:int:lim:rho}
If $f:\T\rightarrow\R$ is ld-continuous and $t\in\T_{\kappa}$,
then
$$\int_{\rho(t)}^{t}f(\tau)\nabla \tau=\nu(t)f(t).$$
\end{theorem}

Properties of the nabla integral, analogous to the ones of the delta integral
given in Theorem~\ref{th:int:prop:delta}, can be found
in Theorem~8.47 of \cite{BohnerDEOTS}. Here we give two
special cases of the nabla integral.

\begin{theorem}[See Theorem 8.48 of \cite{BohnerDEOTS}]
Assume $a,b\in\T$ and $f:\T\lra\R$ is ld-continuous.
\begin{enumerate}
\item If \, $\T=\R$, then
$$
\int\limits_{a}^{b}f(t)\n t=\int\limits_{a}^{b} f(t)dt,
$$
where the integral on the right hand side is the Riemann integral.
\item
If $\T$ consists of only isolated points, then
\begin{equation*}
\int\limits_{a}^{b}f(t)\n t
=
\begin{cases}
\sum\limits_{t\in (a,b]}\nu(t)f(t), & \hbox{ if } a<b, \\
0, & \hbox{ if } a=b,\\
-\sum\limits_{t\in (b,a]}\nu(t)f(t), & \hbox{ if } a>b.
\end{cases}
\end{equation*}
\end{enumerate}
\end{theorem}


\subsection{Relation between delta and nabla operators}
\label{subsec:relation:d:n}

It is possible to relate the approach of Section~\ref{sec:prelim:delta}
with that of Section~\ref{sec:prelim:nabla}.

\begin{theorem}[See Theorems~2.5 and 2.6 of \cite{AticiGreen's_functions}]
\label{th:differ:delta:nabla}
If $f:\mathbb{T}\rightarrow\mathbb{R}$ is delta differentiable on $\mathbb{T}^{\kappa}$
and if $f^{\Delta}$ is continuous on $\mathbb{T}^{\kappa}$,
then $f$ is nabla differentiable on $\mathbb{T}_{\kappa}$ with
\begin{equation*}
f^{\nabla}(t)=\left(f^{\Delta}\right)^{\rho}(t) \, \textrm{ for all } t\in\mathbb{T}_{\kappa}.
\end{equation*}
If $f:\mathbb{T}\rightarrow\mathbb{R}$ is nabla differentiable
on $\mathbb{T}_{\kappa}$ and if $f^{\nabla}$ is continuous on $\mathbb{T}_{\kappa}$,
then $f$ is delta differentiable on $\mathbb{T}^{\kappa}$ with
\begin{equation}
\label{eq:delta:nabla:sigma}
f^{\Delta}(t)=\left(f^{\nabla}\right)^{\sigma}(t) \, \textrm{ for all } t\in\mathbb{T}^{\kappa}.
\end{equation}
\end{theorem}

\begin{theorem}[Proposition~17 of \cite{TorresDeltaNabla}]
\label{th:int:delta:nabla}
If function $f:\mathbb{T}\rightarrow\mathbb{R}$ is continuous,
then for all $a,b\in\mathbb{T}$ with $a<b$ we have
\begin{gather*}
\int\limits_{a}^{b}f(t)\Delta t
=\int\limits_{a}^{b}f^{\rho}(t)\nabla t,\\
\int\limits_{a}^{b}f(t)\nabla t
=\int\limits_{a}^{b}f^{\sigma}(t)\Delta t.
\end{gather*}
\end{theorem}

For a more general theory relating delta and nabla approaches,
we refer the reader to the duality theory of Caputo \cite{cc:dual}.


\section{Direct problems of the calculus of variations on time scales}
\label{CoV}

There are two available approaches to the (direct) calculus of variations on time scales.
The first one, the delta approach, is widely described in literature (see, e.g.,
\cite{BohnerCOVOTS,mb:gg:ap,BohnerDEOTS,MBbook2003,FerreiraTorres,MR2405376,%
HilscgerZeidan,TorresDeltaNabla,Generalizing_the_variational_theory,Torres,Wyrwas}).
The latter one, the nabla approach, was introduced mainly due to its applications in economics
(see, e.g., \cite{MR2218315,AticiGreen's_functions,Atici:comparison,Atici:HMMS}).
It has been shown that these two types of calculus of variations
are dual \cite{cc:dual,MR2957726,MalinowskaTorresCompositionNabla}.


\subsection{The delta approach to the calculus of variations}
\label{sec:calculus:of:variations}

In this section we present the basic information
about the delta calculus of variations on time scales.
Let $\mathbb{T}$ be a given time scale with at least three points,
and $a,b\in\T$, $a<b$, $a=\min\T$ and $b=\max\T$.
Consider the following variational problem on the time scale $\mathbb{T}$:
\begin{equation}
\label{eq:var:probl:delta:sigma}
\mathcal{L}[y]=\int\limits_{a}^{b}
L\left(t, y^{\sigma}(t),y^{\Delta}(t)\right)\Delta t \longrightarrow \min
\end{equation}
subject to the boundary conditions
\begin{equation}
\label{eq:bound:conds:1}
y(a)=y_{a}, \quad y(b)=y_{b}, \quad y_{a},y_{b}\in\R^{n}, \quad n\in\N.
\end{equation}

\begin{definition}
A function $y\in C_{rd}^{1}(\mathbb{T},\R^{n})$ is said to be an admissible path
(function) to problem \eqref{eq:var:probl:delta:sigma}--\eqref{eq:bound:conds:1}
if it satisfies the given boundary conditions $y(a)=y_{a}$, $y(b)=y_{b}$.
\end{definition}

In what follows the Lagrangian
$L$ is understood as a function $L:\T\times\R^{2n}\lra\R$,
$(t,y,v) \rightarrow L(t,y,v)$, and by $L_y$ and $L_v$ we denote
the partial derivatives of $L$ with respect to $y$ and $v$, respectively.
Similar notation is used for second order partial derivatives.
We assume that $L(t,\cdot,\cdot)$ is differentiable in $(y,v)$;
$L(t,\cdot,\cdot)$, $L_{y}(t,\cdot,\cdot)$ and $L_{v}(t,\cdot,\cdot)$
are continuous at $\left(y^{\s}(t),y^{\Delta}(t)\right)$ uniformly at $t$
and rd-continuous at $t$ for any admissible path $y$.
Let us consider the following norm in $C_{rd}^{1}$:
$$
\|y\|_{C^{1}_{rd}}
= \sup_{t\in[a,b]}\|y(t)\|
+\sup_{t\in[a,b]^{\kappa}}\|y^{\triangle}(t)\|,
$$
where $\|\cdot\|$ is the Euclidean norm in $\mathbb{R}^n$.

\begin{definition}
We say that an admissible function $\hat{y}\in C^{1}_{rd}(\T;\R^{n})$
is a local minimizer (respectively, a local maximizer)
to problem \eqref{eq:var:probl:delta:sigma}--\eqref{eq:bound:conds:1}
if there exists $\delta >0$ such that $\mathcal{L}[\hat{y}]
\le\mathcal{L}[y]$ (respectively, $\mathcal{L}[\hat{y}]\geq\mathcal{L}[y]$) for all admissible
functions $y\in C^{1}_{rd}(\T; \R^{n})$ satisfying the inequality $||y-\hat{y}||_{C_{rd}^{1}}<\delta$.
\end{definition}

Local minimizers (or maximizers) to problem \eqref{eq:var:probl:delta:sigma}--\eqref{eq:bound:conds:1}
fulfill the delta differential Euler--Lagrange equation.

\begin{theorem}[Delta differential Euler--Lagrange equation -- see Theorem~4.2 of \cite{BohnerCOVOTS}]
\label{th:EL:1}
If $\hat{y}\in C^{2}_{rd}(\T; \R^{n})$ is a local minimizer
to \eqref{eq:var:probl:delta:sigma}--\eqref{eq:bound:conds:1},
then the Euler--Lagrange equation (in the delta differential form)
\begin{equation*}
L^{\Delta}_{v}\left(t,\hat{y}^{\s}(t),\hat{y}^{\Delta}(t)\right)
=L_{y}\left(t,\hat{y}^{\s}(t),\hat{y}^{\Delta}(t)\right)
\end{equation*}
holds for $t\in [a,b]^{\kappa}$.
\end{theorem}

The next theorem provides the delta integral Euler--Lagrange equation.

\begin{theorem}[Delta integral Euler--Lagrange equation -- see \cite{FerreiraTorres,HilscgerZeidan}]
\label{th:EL:delta:sigma}
If $\hat{y}(t)\in C_{rd}^{1}(\T; \R^{n})$
is a local minimizer of the variational problem \eqref{eq:var:probl:delta:sigma}--\eqref{eq:bound:conds:1},
then there exists a vector $c\in\mathbb{R}^{n}$ such that
the Euler--Lagrange equation (in the delta integral form)
\begin{equation}
\label{eq:EL:delta:sigma}
L_{v}\left(t, \hat{y}^{\sigma}(t),\hat{y}^{\Delta}(t)\right)
=\int\limits_{a}^{t}L_{y}(\tau, \hat{y}^{\sigma}(\tau),\hat{y}^{\Delta}(\tau))\Delta\tau
+c^{T}
\end{equation}
holds for $t\in[a,b]^{\kappa}$.
\end{theorem}

In the proof of Theorem~\ref{th:EL:1} and Theorem~\ref{th:EL:delta:sigma} a time scale
version of the Dubois--Reymond lemma is used.

\begin{lemma}[See \cite{BohnerCOVOTS,MR2405376}]
\label{lem:Dubois:Reymond:delta}
Let $f\in C_{rd}$, $f:[a,b]\lra\R^{n}$. Then
$$
\int\limits_{a}^{b}f^{T}(t)\eta^{\Delta}(t)\Delta t=0
$$
holds for all $\eta \in C^{1}_{rd}([a,b],\mathbb{R}^{n})$ with $\eta(a)=\eta(b)=0$
if and only if $f(t) = c$ for all $t\in [a,b]^{\kappa}$, $c\in\R^{n}$.
\end{lemma}

The next theorem gives a second order necessary optimality condition
for problem \eqref{eq:var:probl:delta:sigma}--\eqref{eq:bound:conds:1}.

\begin{theorem}[Legendre condition -- see Result 1.3 of \cite{BohnerCOVOTS}]
\label{th:Legendre}
If $\hat{y}\in C^{2}_{rd}(\T; \R^{n})$ is a local minimizer of the variational problem
\eqref{eq:var:probl:delta:sigma}--\eqref{eq:bound:conds:1}, then
\begin{equation}
\label{eq:Legendre}
A(t)+\mu(t)\left\lbrace C(t)+C^{T}(t)+\mu(t)B(t)
+(\mu(\sigma(t)))^{\dag}A(\sigma(t))\right\rbrace\geq 0,
\end{equation}
$t\in[a,b]^{\kappa^{2}}$, where
\begin{equation*}
\begin{split}
&
A(t)=L_{vv}\left(t,\hat{y}^{\sigma}(t),\hat{y}^{\Delta}(t)\right),\\
&
B(t)=L_{yy}\left(t,\hat{y}^{\sigma}(t),\hat{y}^{\Delta}(t)\right),\\
&
C(t)=L_{yv}\left(t,\hat{y}^{\sigma}(t),\hat{y}^{\Delta}(t)\right)
\end{split}
\end{equation*}
and where $\alpha^{\dag}=\frac{1}{\alpha}$
if $\alpha\in\mathbb{R}\setminus\lbrace 0 \rbrace$ and $0^{\dag}=0$.
\end{theorem}

\begin{remark}
If \eqref{eq:Legendre} holds with the strict inequality ``$>$'',
then it is called \emph{the strengthened Legendre
condition}.
\end{remark}


\subsection{The nabla approach to the calculus of variations}
\label{subsec:nabla:CoV}

In this section we consider a problem of the calculus of variations
that involves a functional with a nabla derivative and a nabla integral.
The motivation to study such variational problems is coming from applications,
in particular from economics \cite{MR2218315,Atici:HMMS}.
Let $\T$ be a given time scale, which has sufficiently many points
in order for all calculations to make sense, and let $a,b\in\T$, $a<b$.
The problem consists of minimizing or maximizing
\begin{equation}
\label{eq:var:probl:nabla}
\mathcal{L}[y]=\int\limits_{a}^{b}
L\left(t, y^{\rho}(t), y^{\n}(t)\right)\n t
\end{equation}
in the class of functions $y\in C^{1}_{ld}(\T;\R^{n})$ subject to the boundary conditions
\begin{equation}
\label{eq:bound:conds:nabla}
y(a)=y_{a}, \quad y(b)=y_{b}, \quad y_{a}, y_{b}\in\R^{n}, \quad n\in\N.
\end{equation}

\begin{definition}
A function $y\in C_{ld}^{1}(\mathbb{T},\R^{n})$ is said to be an admissible path (function) to problem
\eqref{eq:var:probl:nabla}--\eqref{eq:bound:conds:nabla} if it satisfies
the given boundary conditions $y(a)=y_{a}$, $y(b)=y_{b}$.
\end{definition}
As before, the Lagrangian $L$ is understood as a function $L:\T\times\R^{2n}\lra\R$,
$(t,y,v) \rightarrow L(t,y,v)$.
We assume that $L(t,\cdot,\cdot)$ is differentiable in $(y,v)$;
$L(t,\cdot,\cdot)$, $L_{y}(t,\cdot,\cdot)$ and $L_{v}(t,\cdot,\cdot)$
are continuous at $\left(y^{\rho}(t),y^{\nabla}(t)\right)$ uniformly at $t$
and ld-con\-ti\-nu\-ous at $t$ for any admissible path $y$.
Let us consider the following norm in $C_{ld}^{1}$:
$$
\|y\|_{C^{1}_{ld}}
= \sup_{t\in[a,b]}\|y(t)\|
+\sup_{t\in[a,b]_{\kappa}}\|y^{\nabla}(t)\|
$$
with $\|\cdot\|$ the Euclidean norm in $\mathbb{R}^n$.

\begin{definition}[See \cite{Almeida:iso}]
We say that an admissible function $y\in C^{1}_{ld}(\T;\R^{n})$ is a local minimizer
(respectively, a local maximizer) for the variational problem
\eqref{eq:var:probl:nabla}--\eqref{eq:bound:conds:nabla} if there exists $\delta >0$
such that $\mathcal{L}[\hat{y}]\leq\mathcal{L} [y]$  (respectively,
$\mathcal{L}[\hat{y}]\geq\mathcal{L} [y]$) for all $y\in C^{1}_{ld}(\T; \R^{n})$
satisfying the inequality $||y-\hat{y}||_{C^{1}_{ld}} <\delta$.
\end{definition}

In case of the first order necessary optimality condition
for nabla variational problem on time scales
\eqref{eq:var:probl:nabla}--\eqref{eq:bound:conds:nabla},
the Euler--Lagrange equation takes the following form.

\begin{theorem}[Nabla Euler--Lagrange equation -- see \cite{Torres}]
\label{th:var:probl:nabla}
If a function $\hat{y}\in C_{ld}^{1}(\T;\R^{n})$
provides a local extremum to the variational problem
\eqref{eq:var:probl:nabla}--\eqref{eq:bound:conds:nabla}, then
$\hat{y}$ satisfies the Euler--Lagrange equation (in the nabla differential form)
\begin{equation*}
L^{\nabla}_{v}\left(t, y^{\rho}(t), y^{\n}(t)\right)
=L_{y}\left(t, y^{\rho}(t), y^{\n}(t)\right)
\end{equation*}
for all $t\in [a,b]_{\kappa}$.
\end{theorem}

Now we present the fundamental lemma of the nabla calculus of variations on time scales.

\begin{lemma}[See \cite{natalia:CV}]
\label{lem:Dubois:Reymond:nabla}
Let $f\in C_{ld}([a,b],\mathbb{R}^{n})$. If
$$
\int\limits_{a}^{b}f(t)\eta^{\nabla}(t)\nabla t=0
$$
for all
$\eta \in C^{1}_{ld}([a,b],\mathbb{R}^{n})$ with
$\eta(a)=\eta(b)=0$,
then $f(t)=c$ for all $t\in [a,b]_{\kappa}$, $c\in\R^{n}$.
\end{lemma}

For a good survey on the direct calculus of variations on time scales,
covering both delta and nabla approaches, we refer the reader to \cite{Torres}.


\section{Inverse problems of the calculus of variations on time scales}
\label{chapter_5_P3+P4}

This section is devoted to inverse problems of the calculus of variations on an arbitrary time scale.
To our best knowledge, the inverse problem has not been studied before 2014
\cite{PhD:thesis:Monia,Dryl:Torres:1,MyID:291} in the framework of time scales,
in contrast with the direct problem, that establishes
dynamic equations of Euler--Lagrange type to time-scale variational problems,
that has now been investigated for ten years, since 2004 \cite{BohnerCOVOTS}.
To begin (Section~\ref{P3}) we consider an inverse extremal problem associated
with the following fundamental problem of the calculus of variations: to minimize
\begin{equation}
\label{P3:eq:funct:1}
\mathcal{L}[y]=\int\limits_{a}^{b}
L\left(t,y^{\sigma}(t),y^{\Delta}(t)\right)\Delta t
\end{equation}
subject to boundary conditions $y(a)=y_{0}(a)$, $y(b)=y_{0}(b)$
on a given time scale $\mathbb{T}$.
The Euler--Lagrange equation and the strengthened
Legendre condition are used in order to describe a general form
of a variational functional \eqref{P3:eq:funct:1} that attains an extremum at a given function $y_0$.
In the latter Section~\ref{P4}, we introduce
a completely different approach to the inverse problem of the calculus
of variations, using an integral perspective instead of the classical differential point of view
\cite{Helmholtz,Davis}. We present a sufficient condition
of self-adjointness for an integro-differential equation (Lemma~\ref{P4:lem:suff:self:adj}). Using this property,
we prove a necessary condition for an integro-differential equation
on an arbitrary time scale $\mathbb{T}$ to be an Euler--Lagrange equation (Theorem~\ref{P4:th:necess:EL:int:diff}),
related to a property of self-adjointness (Definition~\ref{P4:def:self:adj:int:diff}) of its equation of variation
(Definition~\ref{P4:def:eq:var}).


\subsection{A general form of the Lagrangian}
\label{P3}

The problem under our consideration is to find a general form of the variational functional
\begin{equation}
\label{P3:eq:var:probl:delta:sigma}
\mathcal{L}[y]=\int\limits_{a}^{b}
L\left(t,y^{\sigma}(t),y^{\Delta}(t)\right)\Delta t
\end{equation}
subject to boundary conditions $y(a)=y(b)=0$,
possessing a local minimum at zero, under the Euler--Lagrange
and the strengthened Legendre conditions.
We assume that $L(t,\cdot,\cdot)$ is a $C^{2}$-function with respect to $(y,v)$
uniformly in $t$, and $L$, $L_{y}$, $L_{v}$, $L_{vv}\in C_{rd}$
for any admissible path $y(\cdot)$.
Observe that under our assumptions, by Taylor's theorem, we may
write $L$, with the big $O$ notation, in the form
\begin{equation}
\label{P3:eq:pre:integrand}
L(t,y,v)=P(t, y) +Q(t, y) v +\frac{1}{2} R(t, y,0)v^{2} + O(v^3),
\end{equation}
where
\begin{equation}
\label{P3:eq:notation:P:Q:R}
\begin{split}
P(t, y) &= L(t, y,0),\\
Q(t, y) &= L_{v}(t, y,0),\\
R(t, y,0) &= L_{vv}(t, y,0).
\end{split}
\end{equation}
Let $R(t, y, v) = R(t, y,0) + O(v)$.
Then, one can write \eqref{P3:eq:pre:integrand} as
\begin{equation*}
L(t, y,v)=P(t, y) +Q(t, y) v +\frac{1}{2} R(t, y, v) v^{2}.
\end{equation*}
Now the idea is to find general forms of $P(t, y^{\sigma}(t))$,
$Q(t, y^{\sigma}(t))$ and $R(t, y^{\sigma}(t), y^{\Delta}(t))$ using the Euler--Lagrange \eqref{eq:EL:delta:sigma}
and the strengthened Legendre \eqref{eq:Legendre} conditions with notation \eqref{P3:eq:notation:P:Q:R}.
Then we use the Euler--Lagrange equation \eqref{eq:EL:delta:sigma} and choose an arbitrary function $P(t,y^{\sigma}(t))$ such that
$P(t,\cdot)\in C^{2}$ with respect to the second variable, uniformly in $t$,
$P$ and $P_y$ rd-continuous in $t$ for all admissible $y$.
We can write the general form of $Q$ as
\begin{equation*}
Q(t,y^{\sigma}(t))=C+\int\limits_{a}^{t}P_{y}(\tau,0)\Delta \tau
+q(t,y^{\sigma}(t))-q(t,0),
\end{equation*}
where $C\in\mathbb{R}$ and $q$ is an arbitrarily function such that $q(t,\cdot)\in C^{2}$
with respect to the second variable, uniformly in $t$, $q$ and $q_y$ are rd-continuous
in $t$ for all admissible $y$.
From the strengthened Legendre condition \eqref{eq:Legendre}, with notation \eqref{P3:eq:notation:P:Q:R},
we set
\begin{equation}
\label{P3:eq:1}
R(t,0,0)+\mu(t)\left\lbrace 2Q_{y}(t,0)+\mu(t)P_{yy}(t,0)
+\left(\mu^{\sigma}(t)\right)^{\dag}R(\sigma(t),0,0)\right\rbrace = p(t)
\end{equation}
with $p\in C_{rd}([a,b])$, $p(t)>0$ for all $t\in [a,b]^{\kappa^{2}}$, chosen arbitrary, where $\alpha^{\dag}=\frac{1}{\alpha}$
if $\alpha\in\mathbb{R}\setminus\lbrace 0 \rbrace$ and $0^{\dag}=0$.
We obtain the following theorem, which presents a general form of the integrand
$L$ for functional \eqref{P3:eq:var:probl:delta:sigma}.

\begin{theorem}
\label{P3:th:integrand:1}
Let $\mathbb{T}$ be an arbitrary time scale. If functional \eqref{P3:eq:var:probl:delta:sigma}
with boundary conditions $y(a)=y(b)=0$ attains a local minimum at $\hat{y}(t)\equiv 0$
under the strengthened Legendre condition, then its Lagrangian $L$ takes the form
\begin{equation*}
\begin{split}
&L\left(t,y^{\sigma}(t),y^{\Delta}(t)\right)
= P\left(t,y^{\sigma}(t)\right)\\
&+\left(C+\int\limits_{a}^{t}P_{y}(\tau,0)\Delta \tau
+q(t,y^{\sigma}(t))-q(t,0)\right)y^{\Delta}(t)\\
&+\Biggl(p(t)-\mu(t)\left\lbrace 2Q_{y}(t,0)+\mu(t) P_{yy}(t,0)
+\left(\mu^{\sigma}(t)\right)^{\dag}R(\sigma(t),0,0)\right\rbrace\\
&\qquad +w(t,y^{\sigma}(t),y^{\Delta}(t))
-w(t,0,0)\Biggr)\frac{(y^{\Delta}(t))^{2}}{2},
\end{split}
\end{equation*}
where $R(t,0,0)$ is a solution of equation \eqref{P3:eq:1},
$C\in\mathbb{R}$, $\alpha^{\dag}=\frac{1}{\alpha}$
if $\alpha\in\mathbb{R}\setminus\lbrace 0 \rbrace$ and $0^{\dag}=0$.
Functions $P$, $p$, $q$ and $w$ are arbitrary functions satisfying:
\begin{itemize}
\item[(i)]\
$P(t,\cdot),q(t,\cdot)\in C^{2}$ with respect to the second variable uniformly in $t$;
$P$, $P_y$, $q$, $q_y$ are rd-continuous in $t$ for all admissible $y$;
$P_{yy}(\cdot,0)$ is rd-continuous in $t$; $p\in C^{1}_{rd}$ with
$p(t)>0$ for all $t\in [a,b]^{\kappa^{2}}$;
\item[(ii)]\
$w(t,\cdot,\cdot)\in C^2$ with respect to the second and the third variable, uniformly in $t$;
$w, w_y, w_v, w_{vv}$ are rd-continuous in $t$ for all admissible $y$.
\end{itemize}
\end{theorem}

\begin{proof}
See \cite{Dryl:Torres:1}.
\end{proof}

Now we consider the general situation when the variational problem consists in minimizing
\eqref{P3:eq:var:probl:delta:sigma} subject to arbitrary boundary conditions $y(a)=y_{0}(a)$ and $y(b)=y_{0}(b)$,
for a certain given function $y_{0}\in C_{rd}^{2}([a,b])$.

\begin{theorem}
\label{P3:th:integrand:2}
Let $\mathbb{T}$ be an arbitrary time scale. If the variational functional
\eqref{P3:eq:var:probl:delta:sigma} with boundary conditions $y(a)=y_{0}(a)$, $y(b)=y_{0}(b)$,
attains a local minimum for a certain given function
$y_{0}(\cdot)\in C^{2}_{rd}([a,b])$
under the strengthened Legendre condition, then its Lagrangian $L$ has the form
\begin{equation*}
\begin{split}
&L\left(t,y^{\sigma}(t),y^{\Delta}(t)\right)
= P\left(t,y^{\sigma}(t)-y^{\sigma}_{0}(t)\right)
+ \left(y^{\Delta}(t)-y_{0}^{\Delta}(t)\right)\\
&\times \left(C+\int\limits_{a}^{t}P_{y}\left(\tau,-y_{0}^{\sigma}(\tau)\right)\Delta \tau
+q\left(t,y^{\sigma}(t)-y^{\sigma}_{0}(t)\right)
-q\left(t,-y_{0}^{\sigma}(t)\right)\right)
+\frac{1}{2}\Biggl(p(t)\\
& -\mu(t) \left\lbrace 2Q_{y}(t,-y_{0}^{\sigma}(t))+\mu(t) P_{yy}(t,-y_{0}^{\sigma}(t))
+\left(\mu^{\sigma}(t)\right)^{\dag}R(\sigma(t),-y_{0}^{\sigma}(t),-y_{0}^{\Delta}(t))\right\rbrace\\
&+w(t,y^{\sigma}(t)-y^{\sigma}_{0}(t),y^{\Delta}(t)-y_{0}^{\Delta}(t))
-w\left(t,-y_{0}^{\sigma}(t),-y_{0}^{\Delta}(t)\right)\Biggr)
\left(y^{\Delta}(t)-y_{0}^{\Delta}(t)\right)^{2},
\end{split}
\end{equation*}
where $C\in\mathbb{R}$ and functions $P$, $p$, $q$, $w$
satisfy conditions (i) and (ii) of Theorem~\ref{P3:th:integrand:1}.
\end{theorem}

\begin{proof}
See \cite{Dryl:Torres:1}.
\end{proof}

For the classical situation $\mathbb{T}=\mathbb{R}$,
Theorem~\ref{P3:th:integrand:2} gives a recent result
of \cite{orlov,MR2907362}.

\begin{corollary}[Theorem~4 of \cite{orlov}]
\label{P3:cor:R}
If the variational functional
$$
\mathcal{L}[y]=\int\limits_{a}^{b} L(t,y(t),y'(t))dt
$$
attains a local minimum at $y_{0}(\cdot)\in C^{2}[a,b]$
when subject to boundary conditions $y(a)=y_{0}(a)$ and $y(b)=y_{0}(b)$
and the classical strengthened Legendre condition
$$
R(t,y_{0}(t),y'_{0}(t))>0, \quad t\in[a,b],
$$
then its Lagrangian $L$ has the form
\begin{equation*}
\begin{split}
&L(t,y(t),y'(t))=P(t,y(t)-y_{0}(t))\\
&
+(y'(t)-y'_{0}(t))\left(C+\int\limits_{a}^{t}
P_{y}(\tau,-y_{0}(\tau))d\tau+q(t,y(t)-y_{0}(t))-q(t,-y_{0}(t))\right)\\
&
+\frac{1}{2}\left(p(t)+w(t,y(t)-y_{0}(t),y'(t)-y'_{0}(t))
-w(t,-y_{0}(t),-y'_{0}(t))\right)(y'(t)-y'_{0}(t))^{2},
\end{split}
\end{equation*}
where $C\in\mathbb{R}$.
\end{corollary}

In the particular case of an isolated time scale, where $\mu(t) \neq 0$ for all $t\in\mathbb{T}$,
we get the following corollary.

\begin{corollary}
\label{P3:cor:isolated}
Let $\mathbb{T}$ be an isolated time scale. If functional \eqref{P3:eq:var:probl:delta:sigma}
subject to the boundary conditions $y(a)=y(b)=0$ attains a local minimum at
$\hat{y}(t) \equiv 0$ under the strengthened Legendre condition,
then the Lagrangian $L$ has the form
\begin{equation*}
\begin{split}
&L\left(t,y^{\sigma}(t),y^{\Delta}(t)\right)
= P\left(t,y^{\sigma}(t)\right)\\
&+\left(C+\int\limits_{a}^{t}P_{y}(\tau,0)\Delta \tau
+q(t,y^{\sigma}(t))-q(t,0)\right)y^{\Delta}(t)\\
&+\left(e_{r}(t,a)R_{0}
+\int\limits_{a}^{t}e_{r}(t,\sigma(\tau))s(\tau)\Delta\tau
+w(t,y^{\sigma}(t),y^{\Delta}(t))
-w(t,0,0)\right)\frac{(y^{\Delta}(t))^{2}}{2},
\end{split}
\end{equation*}
where $C,R_{0}\in\mathbb{R}$ and $r(t)$ and $s(t)$ are given by
\begin{equation}
\label{P3:eq:funct:r:s}
r(t) := -\frac{1+\mu(t)(\mu^{\sigma}(t))^{\dag}}{\mu^{2}(t)(\mu^{\sigma}(t))^{\dag}},
\quad s(t) := \frac{p(t)
-\mu(t)[2Q_{y}(t,0)+\mu(t)P_{yy}(t,0)]}{\mu^{2}(t)(\mu^{\sigma}(t))^{\dag}}
\end{equation}
with $\alpha^{\dag} = \frac{1}{\alpha}$ if $\alpha \in\mathbb{R}\setminus\lbrace 0 \rbrace$ and $0^{\dag}=0$,
and functions $P$, $p$, $q$, $w$ satisfy assumptions of Theorem~\ref{P3:th:integrand:1}.
\end{corollary}

Based on Corollary~\ref{P3:cor:isolated}, we present the form of Lagrangian $L$
in the periodic time scale $\mathbb{T}=h\mathbb{Z}$.

\begin{example}
\label{P3:ex:hZ}
Let $\mathbb{T}=h\mathbb{Z}$, $h > 0$,
and $a, b\in h\mathbb{Z}$ with $a<b$.
Then $\mu(t) \equiv h$.
Consider the variational functional
\begin{equation}
\label{P3:eq:funct:hZ}
\mathcal{L}[y]=h\sum_{k=\frac{a}{h}}^{\frac{b}{h}-1}
L\left(kh,y(kh+h),\Delta_h y(kh)\right)
\end{equation}
subject to the boundary conditions $y(a)=y(b)=0$, which
attains a local minimum at $\hat{y}(kh)\equiv 0$
under the strengthened Legendre condition
$$
R(kh,0,0)+2hQ_{y}(kh,0)+h^{2}P_{yy}(kh,0)+R(kh+h,0,0)>0,
$$
$kh\in [a,b-2h]\cap h\mathbb{Z}$.
Functions $r(t)$ and $s(t)$  (see \eqref{P3:eq:funct:r:s}) have the following form:
\begin{equation*}
r(t) \equiv -\frac{2}{h},
\quad s(t)=\frac{p(t)}{h}
-\left(2Q_{y}(t,0) + h P_{yy}(t,0)\right).
\end{equation*}
Hence,
$$
\int\limits_{a}^{t}P_{y}(\tau,0)\Delta \tau
=h\sum\limits_{i=\frac{a}{h}}^{\frac{t}{h}-1}P_{y}(ih,0),
$$
\begin{equation*}
\int\limits_{a}^{t}e_{r}(t,\sigma(\tau))s(\tau)\Delta \tau
=\sum_{i=\frac{a}{h}}^{\frac{t}{h}-1}(-1)^{\frac{t}{h}-i-1}
\left(p(ih)-2hQ_{y}(ih,0)-h^{2}P_{yy}(ih,0)\right).
\end{equation*}
Thus, the Lagrangian $L$ of the variational functional \eqref{P3:eq:funct:hZ}
on $\mathbb{T}=h\mathbb{Z}$ has the form
\begin{equation*}
\begin{split}
L&\left(kh,y(kh+h),\Delta_h y(kh)\right)=P\left(kh,y(kh+h)\right)\\
&+\left(C+\sum\limits_{i=\frac{a}{h}}^{k-1}hP_{y}(ih,0)
+q(kh,y(kh+h))-q(kh,0)\right)\Delta_h y(kh)\\
&+\frac{1}{2}\Biggl((-1)^{k-\frac{a}{h}}R_{0}
+\sum_{i=\frac{a}{h}}^{k-1}(-1)^{k-i-1}
\left(p(i h)-2hQ_{y}(ih,0)-h^{2}P_{yy}(ih,0)\right)\\
&\qquad +w(kh,y(kh+h),\Delta_h y(kh))-w(kh,0,0)\Biggr)
\left(\Delta_h y(kh)\right)^{2},
\end{split}
\end{equation*}
where functions $P$, $p$, $q$, $w$ are arbitrary
but satisfy assumptions of Theorem~\ref{P3:th:integrand:1}.
\end{example}


\subsection{Necessary condition for an Euler--Lagrange equation}
\label{P4}

This section provides a necessary condition for an integro-differential equation
on an arbitrary time scale to be an Euler--Lagrange equation (Theorem~\ref{P4:th:necess:EL:int:diff}).
For that the notions of self-adjointness (Definition~\ref{P4:def:self:adj:int:diff})
and equation of variation (Definition~\ref{P4:def:eq:var}) are essential.

\begin{definition}[First order self-adjoint integro-differential equation]
\label{P4:def:self:adj:int:diff}
A first order integro-differential dynamic equation is said to be \emph{self-adjoint} if it has the form
\begin{equation}
\label{P4:eq:self:adj:int:diff}
Lu(t)=const, \hbox{ where } Lu(t)=p(t)u^{\Delta}(t)+\int\limits_{t_{0}}^{t}\left[r(s)u^{\sigma}(s)\right]\Delta s
\end{equation}
with $p,r\in C_{rd}$, $p\neq 0$ for all $t\in\mathbb{T}$ and $t_{0}\in\T$.
\end{definition}

Let $\mathbb{D}$ be the set of all functions $y:\mathbb{T}\rightarrow \mathbb{R}$ such that
$y^{\Delta}:\mathbb{T}^{\kappa}\rightarrow\mathbb{R}$ is continuous.
A function $y\in\mathbb{D}$ is said to be a solution of \eqref{P4:eq:self:adj:int:diff}
provided $Ly(t)=const$  holds for all $t\in\mathbb{T^{\kappa}}$.
For simplicity, we use the operators $[\cdot]$ and $\langle \cdot\rangle$ defined as
\begin{equation}
\label{eq:2convenientoper:++}
[y](t):=(t, y^{\s}(t), y^{\Delta}(t)), \quad\quad \langle y\rangle (t):=(t, y^{\s}(t), y^{\Delta}(t), y^{\Delta\Delta}(t)),
\end{equation}
and partial derivatives of function $(t,y,v,z)\lra L(t,y,v,z)$ are denoted by
$\partial_{2}L=L_{y}$, $\partial_{3}L=L_{v}$, $\partial_{4}L=L_{z}$.

\bigskip
\bigskip

\begin{definition}[Equation of variation]
\label{P4:def:eq:var}
Let
\begin{equation}
\label{P4:eq:int:diff:1}
H[y](t)+\int\limits_{t_{0}}^{t}G[y](s)\Delta s = const
\end{equation}
be an integro-differential equation on time scales with
$H_{v}\neq 0$, $t\lra F_{y}[y](t)$, $t\lra F_{v}[y](t)
\in C_{rd}(\T,\R)$ along every curve $y$, where $F\in\lb G,H\rb$.
The \emph{equation of variation}
associated with \eqref{P4:eq:int:diff:1} is given by
\begin{equation}
\label{P4:eq:eq:var}
H_{y}[u](t)u^{\sigma}(t)+H_{v}[u](t) u^{\Delta}(t)
+\int\limits_{t_{0}}^{t}G_{y}[u](s)u^{\sigma}(s)+G_{v}[u](s) u^{\Delta}(s)\Delta s=0.
\end{equation}
\end{definition}

\begin{lemma}[Sufficient condition of self-adjointness]
\label{P4:lem:suff:self:adj}
Let \eqref{P4:eq:int:diff:1} be a given integro-differential equation. If
\begin{equation*}
H_{y}[y](t)+G_{v}[y](t)=0,
\end{equation*}
then its equation of variation \eqref{P4:eq:eq:var} is self-adjoint.
\end{lemma}

\begin{proof}
See \cite{MyID:291}.
\end{proof}

Now we provide an answer to the general inverse problem of the calculus of variations on time scales.

\begin{theorem}[Necessary condition for an Euler--Lagrange equation in integral form]
\label{P4:th:necess:EL:int:diff}
Let $\mathbb{T}$ be an arbitrary time scale and
\begin{equation}
\label{P4:eq:int:diff:2}
H(t,y^{\sigma}(t),y^{\Delta}(t))+\int\limits_{t_{0}}^{t}G(s,y^{\sigma}(s),y^{\Delta}(s))\Delta s=const
\end{equation}
be a given integro-differential equation. If \eqref{P4:eq:int:diff:2}
is to be an Euler--Lagrange equation, then its equation of variation
\eqref{P4:eq:eq:var} is self-adjoint
in the sense of Definition~\ref{P4:def:self:adj:int:diff}.
\end{theorem}

\begin{proof}
See \cite{MyID:291}.
\end{proof}

\begin{remark}
\label{P4:rem:not:EL}
In practical terms, Theorem~\ref{P4:th:necess:EL:int:diff} is useful to identify equations
that are not Euler--Lagrange: if the equation of variation \eqref{P4:eq:eq:var}
of a given dynamic equation \eqref{P4:eq:int:diff:1} is not self-adjoint,
then we conclude that \eqref{P4:eq:int:diff:1} is not an Euler--Lagrange equation.
\end{remark}

Now we present an example of a second order differential equation
on time scales which is not an Euler--Lagrange equation.

\begin{example}
Let us consider the following second-order linear oscillator
dynamic equation on an arbitrary time scale $\mathbb{T}$:
\begin{equation}
\label{P4:eq:3}
y^{\Delta\Delta}(t)+y^{\Delta}(t)-t=0.
\end{equation}
We may write equation \eqref{P4:eq:3}
in integro-differential form \eqref{P4:eq:int:diff:1}:
\begin{equation}
\label{P4:eq:4}
y^{\Delta}(t)+\int\limits_{t_{0}}^{t}\left(y^{\Delta}(s)-s\right)\Delta s=const,
\end{equation}
where $H[y](t)=y^{\Delta}(t)$ and $G[y](t)=y^{\Delta}(t)-t$.
Because
\begin{equation*}
H_{y}[y](t)=G_{y}[y](t)=0,
\quad
H_{v}[y](t)=G_{v}[y](t)=1,
\end{equation*}
the equation of variation associated with \eqref{P4:eq:4} is given by
\begin{equation}
\label{P4:eq:5}
u^{\Delta}(t)+\int\limits_{t_{0}}^{t}u^{\Delta}(s)\Delta s=0 \iff u^{\Delta}(t)+u(t)=u(t_{0}).
\end{equation}
We may notice that equation \eqref{P4:eq:5} cannot be written in form
\eqref{P4:eq:self:adj:int:diff}, hence, it is not self-adjoint.
Following Theorem~\ref{P4:th:necess:EL:int:diff} (see Remark~\ref{P4:rem:not:EL})
we conclude that equation \eqref{P4:eq:3} is not an Euler--Lagrange equation.
\end{example}

Now we consider the particular case of Theorem~\ref{P4:th:necess:EL:int:diff}
when $\mathbb{T}=\mathbb{R}$ and $y\in C^{2}([t_{0},t_{1}];\mathbb{R})$. In this case
operator $[\cdot]$ of \eqref{eq:2convenientoper:++} has the form
$$
[y](t)=(t,y(t),y'(t))=:[y]_{\R}(t),
$$
while condition \eqref{P4:eq:self:adj:int:diff} can be written as
\begin{equation}
\label{P4:eq:self:adj:R}
p(t)u'(t)+\int\limits_{t_{0}}^{t}r(s)u(s)ds=const.
\end{equation}

\begin{corollary}
\label{P4:cor:R}
If a given integro-differential equation
$$
H(t,y(t),y'(t))+\int\limits_{t_{0}}^{t}G(s,y(s),y'(s))ds=const
$$
is to be the Euler--Lagrange equation of the variational problem
\begin{equation*}
\label{P4:eq:funct:R}
\mathcal{I}[y]=\int\limits_{t_{0}}^{t_{1}} L(t,y(t),y'(t))dt
\end{equation*}
(cf., e.g., \cite{MR2004181}), then its equation of variation
$$
H_{y}[u]_{\mathbb{R}}(t)u(t)+H_{v}[u]_{\mathbb{R}}(t)u'(t)
+\int\limits_{t_{0}}^{t}G_{y}[u]_{\mathbb{R}}(s)u(s)+G_{v}[u]_{\mathbb{R}}(s)u'(s)ds=0
$$
must be self-adjoint, in the sense of Definition~\ref{P4:def:self:adj:int:diff}
with \eqref{P4:eq:self:adj:int:diff} given by \eqref{P4:eq:self:adj:R}.
\end{corollary}

\begin{proof}
Follows from Theorem \ref{P4:th:necess:EL:int:diff} with $\mathbb{T}=\mathbb{R}$.
\end{proof}

Now we consider the particular case of Theorem~\ref{P4:th:necess:EL:int:diff}
when $\mathbb{T}=h\mathbb{Z}$, $h>0$. In this case operator $[\cdot]$
of \eqref{eq:2convenientoper:++} has the form
$$
[y](t)=(t,y(t+h),\Delta_{h} y(t))=:[y]_{h}(t),
$$
where
$$
\Delta_{h}y(t)=\frac{y(t+h)-y(t)}{h}.
$$
For $\mathbb{T}=h\mathbb{Z}$, $h>0$,
condition \eqref{P4:eq:self:adj:int:diff} can be written as
\begin{equation}
\label{P4:eq:self:adj:hZ}
p(t)\Delta_{h}u(t)+\sum\limits_{k=\frac{t_{0}}{h}}^{\frac{t}{h}-1}hr(kh)u(kh+h)=const.
\end{equation}

\begin{corollary}
\label{P4:cor:hZ}
If a given difference equation
$$
H(t,y(t+h),\Delta_{h} y(t))
+\sum\limits_{k=\frac{t_{0}}{h}}^{\frac{t}{h}-1}hG(kh,y(kh+h),\Delta_{h} y(kh))=const
$$
is to be the Euler--Lagrange equation of the discrete variational problem
\begin{equation*}
\label{P4:eq:funct:hZ}
\mathcal{I}[y]=\sum\limits_{k=\frac{t_{0}}{h}}^{\frac{t_{1}}{h}-1}hL(kh,y(kh+h),\Delta_{h} y(kh))
\end{equation*}
(cf., e.g., \cite{MyID:179}), then its equation of variation
\begin{multline*}
H_{y}[u]_{h}(t)u(t+h)+H_{v}[u]_{h}(t)\Delta_{h}u(t)\\
+h\sum\limits_{k=\frac{t_{0}}{h}}^{\frac{t}{h}-1}\left(
G_{y}[u]_{h}(kh)u(kh+h)+G_{v}[u]_{h}(kh)\Delta_{h}u(kh)\right)=0
\end{multline*}
is self-adjoint, in the sense of Definition~\ref{P4:def:self:adj:int:diff}
with \eqref{P4:eq:self:adj:int:diff} given by \eqref{P4:eq:self:adj:hZ}.
\end{corollary}

\begin{proof}
Follows from Theorem~\ref{P4:th:necess:EL:int:diff} with $\mathbb{T}=h\mathbb{Z}$.
\end{proof}

Finally, let us consider the particular case of Theorem~\ref{P4:th:necess:EL:int:diff}
when $\mathbb{T}=\overline{q^{\mathbb{Z}}}=q^{\mathbb{Z}}\cup\left\{0\right\}$,
where $q^{\mathbb{Z}}=\left\{q^{k}: k\in\mathbb{Z}, q>1\right\}$.
In this case operator $[\cdot]$ of \eqref{eq:2convenientoper:++} has the form
$$
[y]_{\overline{q^{\mathbb{Z}}}}(t)=(t,y(qt),\Delta_{q} y(t))=:[y]_{q}(t),
$$
where
$$
\Delta_{q}y(t)=\frac{y(qt)-y(t)}{(q-1)t}.
$$
For $\mathbb{T}=\overline{q^{\mathbb{Z}}}$, $q>1$,
condition \eqref{P4:eq:self:adj:int:diff} can be written
as (cf., e.g., \cite{Rahmat}):
\begin{equation}
\label{P4:eq:self:adj:qZ}
p(t)\Delta_{q}u(t)+ (q-1)\sum\limits_{s\in [t_{0},t) \cap\mathbb{T}}sr(s)u(qs)=const.
\end{equation}

\begin{corollary}
\label{P4:cor:qZ}
If a given $q$-equation
$$
H(t,y(qt),\Delta_{q} y(t))+(q-1)\sum\limits_{s\in [t_{0},t)
\cap\mathbb{T}}sG(s,y(qs),\Delta_{q}y(s))=const,
$$
$q>1$, is to be the Euler--Lagrange equation of the variational problem
\begin{equation*}
\label{P4:eq:funct:qZ}
\mathcal{I}[y]=(q-1)\sum\limits_{t\in [t_{0},t_{1})
\cap\mathbb{T}}tL(t,y(qt),\Delta_{q}y(t)),
\end{equation*}
$t_{0}, t_{1}\in \overline{q^{\mathbb{Z}}}$, then its equation of variation
\begin{multline*}
H_{y}[ u]_{q}(t)u(qt)+H_{v}[u]_{q}(t)\Delta_{q}u(t)\\
+(q-1)\sum\limits_{s\in [t_{0},t)\cap\mathbb{T}} s\left(
G_{y}[u]_{q}(s)u(qs)+G_{v}[u]_{q}(s)\Delta_{q}u(s) \right)=0
\end{multline*}
is self-adjoint, in the sense of Definition~\ref{P4:def:self:adj:int:diff}
with \eqref{P4:eq:self:adj:int:diff} given by \eqref{P4:eq:self:adj:qZ}.
\end{corollary}

\begin{proof}
Choose $\mathbb{T}=\overline{q^{\mathbb{Z}}}$
in Theorem~\ref{P4:th:necess:EL:int:diff}.
\end{proof}

More information about Euler--Lagrange equations for $q$-variational
problems may be found in \cite{FerreiraTorres,MyID:266,MR2966852}
and references therein.


\subsection{Discussion}
\label{final remarks}

On an arbitrary time scale $\mathbb{T}$, we can easily show equivalence between
the integro-differential equation \eqref{P4:eq:int:diff:1} and the second order differential
equation \eqref{P4:eq:6} below (Proposition~\ref{P4:prop:1}). However, when we consider equations
of variations of them, we notice that it is not possible to prove an equivalence between
them on an arbitrary time scale. The main reason of this impossibility, even in the discrete time scale
$\mathbb{Z}$, is the absence of a general chain rule
on an arbitrary time scale (see, e.g., Example~1.85 of \cite{BohnerDEOTS}).
However, on $\mathbb{T}=\mathbb{R}$ we can present
this equivalence (Proposition~\ref{P4:prop:2}).

\begin{proposition}[See \cite{MyID:291}]
\label{P4:prop:1}
The integro-differential equation \eqref{P4:eq:int:diff:1}
is equivalent to a second order delta differential equation
\begin{equation}
\label{P4:eq:6}
W\left(t,y^{\sigma}(t), y^{\Delta}(t), y^{\Delta\Delta}(t)\right)=0.
\end{equation}
\end{proposition}

Let $\mathbb{T}$ be a time scale such that $\mu$ is delta differentiable.
The equation of variation of a second order differential
equation \eqref{P4:eq:6} is given by
\begin{equation}
\label{P4:eq:eq:var:1}
W_{z}\langle u\rangle (t)u^{\Delta\Delta}(t)
+W_{v}\langle u\rangle (t) u^{\Delta}(t)
+W_{y}\langle u\rangle (t) u^{\sigma}(t)=0.
\end{equation}
On an arbitrary time scale it is impossible to prove the equivalence
between the equation of variation \eqref{P4:eq:eq:var} and \eqref{P4:eq:eq:var:1}.
Indeed, after differentiating both sides of equation \eqref{P4:eq:eq:var}
and using the product rule given by Theorem~\ref{th:differ:prop:delta}, one has
\begin{multline}
\label{P4:eq:9}
H_{y}[u](t)u^{\sigma\Delta}(t)+H_{y}^{\Delta}[u](t)u^{\sigma\sigma}(t)
+H_{v}[u](t)u^{\Delta\Delta}(t)+H_{v}^{\Delta}[u](t)u^{\Delta\sigma}(t)\\
+G_{y}[u](t)u^{\sigma}(t)+G_{v}[u](t)u^{\Delta}(t)=0.
\end{multline}
The direct calculations
\begin{itemize}
\item $H_{y}[u](t)u^{\sigma\Delta}(t)=H_{y}[u](t)(u^{\Delta}(t)
+\mu^{\Delta}(t)u^{\Delta}(t)+\mu^{\sigma}(t)u^{\Delta\Delta}(t))$,

\item $H_{y}^{\Delta}[u](t)u^{\sigma\sigma}(t)
=H_{y}^{\Delta}[u](t)(u^{\sigma}(t)+\mu^{\sigma}(t)u^{\Delta}(t)
+\mu(t)\mu^{\sigma}(t)u^{\Delta\Delta}(t))$,

\item $H_{v}^{\Delta}[u](t)u^{\Delta\sigma}(t)
=H_{v}^{\Delta}[u](t)(u^{\Delta}(t)+\mu u^{\Delta\Delta}(t))$,
\end{itemize}
and the fourth item of Theorem~\ref{th:differ:delta},
allow us to write equation \eqref{P4:eq:9} in form
\begin{multline}
\label{P4:eq:11}
u^{\Delta\Delta}(t)\left[\mu(t)H_{y}[u](t)+H_{v}[u](t)\right]^{\sigma}\\
+u^{\Delta}(t)\left[H_{y}[u](t)+(\mu(t)H_{y}[u](t))^{\Delta}
+H_{v}^{\Delta}[u](t)+G_{v}[u](t)\right]\\
+u^{\sigma}(t)\left[H_{y}^{\Delta}[u](t)+G_{y}[u](t)\right]=0.
\end{multline}
We are not able to prove that the coefficients of
\eqref{P4:eq:11} are the same as in \eqref{P4:eq:eq:var:1},
respectively. This is due to the fact that we cannot
find the partial derivatives of \eqref{P4:eq:6}, that is,
$W_{z}\langle u\rangle(t)$, $W_{v}\langle u\rangle(t)$ and $W_{y}\langle u\rangle(t)$,
from equation \eqref{P4:eq:eq:var:1} because of lack of a general
chain rule in an arbitrary time scale \cite{BohGus1}.
The equivalence, however, is true for $\mathbb{T}=\mathbb{R}$.
Operator $\left\langle \cdot\right\rangle$ has in this case the form
$\left\langle y\right\rangle (t)=(t, y(t), y'(t), y''(t))
=: \left\langle y\right\rangle_{\R} (t)$.

\begin{proposition}[See \cite{MyID:291}]
\label{P4:prop:2}
The equation of variation
\begin{equation*}
H_{y}[u]_{\mathbb{R}}(t)u(t)+H_{v}[u]_{\mathbb{R}}(t)u'(t)
+\int\limits_{t_{0}}^{t}G_{y}[u]_{\mathbb{R}}(s)u(s)
+G_{v}[u]_{\mathbb{R}}(s)u'(s)ds=0
\end{equation*}
is equivalent to the second order differential equation
\begin{equation*}
W_{z}\langle u\rangle_{\mathbb{R}}(t)u''(t)
+W_{v}\langle u\rangle_{\mathbb{R}}(t) u'(t)
+W_{y}\langle u\rangle_{\mathbb{R}}(t)u(t)=0.
\end{equation*}
\end{proposition}

Proposition~\ref{P4:prop:2} allows us to obtain the classical result of
\cite[Theorem II]{Davis} as a corollary of our Theorem~\ref{P4:th:necess:EL:int:diff}.
The absence of a chain rule on an arbitrary time scale (even for $\mathbb{T}=\mathbb{Z}$)
implies that the classical approach \cite{Davis} fails on time scales.
This is the reason why we use here a completely different approach to the subject based
on the integro-differential form. The case $\mathbb{T}=\mathbb{Z}$ was recently investigated
in \cite{Helmholtz}. However, similarly to \cite{Davis}, the approach of \cite{Helmholtz}
is based on the differential form and cannot be extended to general time scales.


\section{The delta-nabla calculus of variations for composition functionals}
\label{sec:mr}

The delta-nabla calculus of variations has been introduced in \cite{TorresDeltaNabla}.
Here we investigate more general problems of the time-scale calculus of variations for a functional
that is the composition of a certain scalar function with the delta and nabla integrals of a vector valued field.
We begin by proving general Euler--Lagrange equations in integral form (Theorem~\ref{P6:th:main}).
Then we consider cases when initial or terminal boundary conditions are not specified,
obtaining corresponding transversality conditions
(Theorems~\ref{P6:th:trans:initial} and \ref{P6:th:trans:terminal}).
Furthermore, we prove necessary optimality conditions for general isoperimetric problems
given by the composition of delta-nabla integrals (Theorem~\ref{P6:th:conds:iso}).
Finally, some illustrating examples are presented (Section~\ref{sec:examples}).


\subsection{The Euler--Lagrange equations}
\label{subsec:EL}

Let us begin by defining the class of functions $C_{k,n}^{1}([a,b];\mathbb{R})$,
which contains delta and nabla differentiable functions.
\begin{definition}
\label{P6:class}
By $C_{k,n}^{1}([a,b];\mathbb{R})$, $k,n\in\N$, we denote the class of functions
$y:[a,b]\rightarrow\mathbb{R}$ such that: if $k\neq 0$ and $n\neq 0$, then
$y^{\Delta}$ is continuous on $[a,b]^{\kappa}_{\kappa}$ and
$y^{\nabla}$ is continuous on $[a,b]_{\kappa}^{\kappa}$,
where $[a,b]^{\kappa}_{\kappa}:=[a,b]^{\kappa}\cap [a,b]_{\kappa}$;
if $n=0$, then $y^{\Delta}$ is continuous on $[a,b]^{\kappa}$;
if $k=0$, then $y^{\nabla}$ is continuous on $[a,b]_{\kappa}$.
\end{definition}
Our aim is to find a function $y$ which minimizes
or maximizes the following variational problem:
\begin{multline}
\label{P6:eq:problem}
\mathcal L[y]=H\left(\int\limits_{a}^{b}f_{1}(t,y^{\sigma}(t),y^{\Delta}(t))\Delta t,
\ldots,\int\limits_{a}^{b}f_{k}(t,y^{\sigma}(t),y^{\Delta}(t))\Delta t,\right.\\
\left.\int\limits_{a}^{b}f_{k+1}(t,y^{\rho}(t),y^{\nabla}(t))\nabla t,\ldots,
\int\limits_{a}^{b}f_{k+n}(t,y^{\rho}(t),y^{\nabla}(t))\nabla t\right),
\end{multline}
\begin{equation}
\label{P6:eq:bound:conds}
(y(a)=y_{a}), \quad (y(b)=y_{b}).
\end{equation}
The parentheses in \eqref{P6:eq:bound:conds}, around the end-point conditions,
means that those conditions may or may not occur
(it is possible that one or both $y(a)$ and $y(b)$ are free).
A function $y\in C_{k,n}^{1}$ is said to be admissible provided
it satisfies the boundary conditions \eqref{P6:eq:bound:conds} (if any is given).
For $k = 0$ problem \eqref{P6:eq:problem}--\eqref{P6:eq:bound:conds} becomes a nabla problem
(neither delta integral nor delta derivative is present);
for $n = 0$ problem \eqref{P6:eq:problem}--\eqref{P6:eq:bound:conds}
reduces to a delta problem (neither nabla integral nor nabla derivative is present).
For simplicity, we use the operators $[\cdot]$ and $\lb\cdot\rb$ defined by
\begin{equation*}
[y](t):=(t,y^{\s}(t),y^{\Delta}(t)),\quad
\lb y\rb (t):=(t,y^{\rho}(t),y^{\n}(t)).
\end{equation*}
We assume that:
\begin{enumerate}
\item the function $H:\mathbb{R}^{n+k}\rightarrow\mathbb{R}$
has continuous partial derivatives with respect to its arguments,
which we denote by $H_{i}^{'}$, $i=1, \ldots, n+k$;

\item functions $(t,y,v)\rightarrow f_{i}(t,y,v)$
from $[a,b]\times\mathbb{R}^{2}$ to $\mathbb{R}$, $i=1,\ldots, n+k$,
have partial continuous derivatives with respect to $y$ and $v$
uniformly in $t \in [a,b]$, which we denote by $f_{iy}$ and $f_{iv}$;

\item $f_{i}$, $f_{iy}$, $f_{iv}$ are rd-continuous
on $[a,b]^{\kappa}$, $i=1,\ldots,k$, and ld-continuous on
$[a,b]_{\kappa}$, $i=k+1,\ldots,k+n$, for all $y\in C_{k,n}^{1}$.
\end{enumerate}

\begin{definition}[Cf. \cite{TorresDeltaNabla}]
\label{def:loc:extr}
We say that an admissible function $\hat{y}\in C_{k,n}^{1}([a,b];\mathbb{R})$
is a local minimizer (respectively, local maximizer) to problem
\eqref{P6:eq:problem}--\eqref {P6:eq:bound:conds}, if there exists
$\delta >0$ such that $\mathcal{L}[\hat{y}]\leq \mathcal{L} [y]$
(respectively, $\mathcal{L}[\hat{y}]\geq \mathcal{L}[y]$)
for all admissible functions $y\in C_{k,n}^{1}([a,b];\mathbb{R})$ satisfying
the inequality $|| y-\hat{y}||_{1,\infty}<\delta$, where
\begin{equation*}
||y||_{1,\infty}:=||y^{\sigma}||_{\infty}+||y^{\Delta}||_{\infty}
+||y^{\rho}||_{\infty}+||y^{\nabla}||_{\infty}
\end{equation*}
with $||y||_{\infty}:= \sup_{t\in[a,b]_{\kappa}^{\kappa}} |y(t)|$.
\end{definition}

For brevity, in what follows we omit the argument of $H_{i}^{'}$. Precisely,
$$
H_{i}^{'}:=\frac{\partial H}{\partial
\mathcal{F}_{i}}(\mathcal{F}_{1}(y),\ldots,\mathcal{F}_{k+n}(y)),
\quad i=1,\ldots,n+k,
$$
where
\begin{equation*}
\begin{split}
\mathcal{F}_{i}(y)
&=\int\limits_{a}^{b} f_{i}[y](t)\Delta t, \hbox{ for }
i=1,\ldots,k,\\
\mathcal{F}_{i}(y)
&=\int\limits_{a}^{b}f_{i}\lb y\rb(t)\nabla t, \hbox{ for } i=k+1,\ldots,k+n.
\end{split}
\end{equation*}
Depending on the given boundary conditions, we can distinguish four different problems.
The first one is the problem $(P_{ab})$, where the two boundary conditions are specified.
To solve this problem we need an Euler--Lagrange necessary optimality condition, which
is given by Theorem~\ref{P6:th:main} below.
Next two problems --- denoted by $(P_{a})$ and $(P_{b})$ --- occur when
$y(a)$ is given and $y(b)$ is free (problem $(P_{a})$) and when
$y(a)$ is free and $y(b)$ is specified (problem $(P_{b})$).
To solve both of them we need an Euler--Lagrange equation
and one proper transversality condition. The last problem --- denoted by  $(P)$ ---
occurs when both boundary conditions are not present.
To find a solution for such a problem we need to use an Euler--Lagrange equation
and two transversality conditions (one at each time $a$ and $b$).

\begin{theorem}[The Euler--Lagrange equations in integral form]
\label{P6:th:main}
If $\hat{y}$ is a local solution
to problem \eqref{P6:eq:problem}--\eqref{P6:eq:bound:conds},
then the Euler--Lagrange equations (in integral form)
\begin{multline}
\label{P6:eq:EL:nabla}
\sum\limits_{i=1}^{k}H_{i}^{'}\cdot
\left(f_{iv}[\hat{y}](\rho(t))-\int\limits_{a}^{\rho(t)}
f_{iy}[\hat{y}](\tau)\Delta \tau \right)\\
+\sum\limits_{i=k+1}^{k+n}H_{i}^{'}\cdot
\left(f_{iv}\lb\hat{y}\rb(t)-\int\limits_{a}^{t}
f_{iy}\lb\hat{y}\rb(\tau)\nabla \tau \right) = c,
\quad  t\in\mathbb{T}_{\kappa},
\end{multline}
and
\begin{multline}
\label{P6:eq:EL:delta}
\sum\limits_{i=1}^{k}H_{i}^{'}\cdot
\left(f_{iv}[\hat{y}](t)-\int\limits_{a}^{t}
f_{iy}[\hat{y}](\tau)\Delta \tau\right)\\
+\sum\limits_{i=k+1}^{k+n}H_{i}^{'}\cdot
\left(f_{iv}\lb\hat{y}\rb(\sigma(t))-\int\limits_{a}^{\sigma(t)}
f_{iy}\lb\hat{y}\rb(\tau)\nabla \tau\right) = c,
\quad  t\in\mathbb{T}^{\kappa},
\end{multline}
hold.
\end{theorem}

\begin{proof}
See \cite{MR3040923}.
\end{proof}

For regular time scales (Definition~\ref{def:regular}),
the Euler--Lagrange equations \eqref{P6:eq:EL:nabla} and \eqref{P6:eq:EL:delta}
coincide; on a general time scale, they are different.
Such a difference is illustrated in Example~\ref{P6:ex:5}.
For such purpose let us define $\xi$ and $\chi$ by
\begin{equation}
\label{P6:eq:xi:chi}
\begin{gathered}
\xi(t):=\sum\limits_{i=1}^{k}H_{i}^{'}\cdot
\left(f_{iv}[\hat{y}](t)-\int\limits_{a}^{t}
f_{iy}[\hat{y}](\tau)\Delta \tau \right),\\
\chi(t):=\sum\limits_{i=k+1}^{k+n}H_{i}^{'}\cdot
\left(f_{iv}\lb\hat{y}\rb(t)-\int\limits_{a}^{t}
f_{iy}\lb\hat{y}\rb(\tau)\nabla \tau \right).
\end{gathered}
\end{equation}

\begin{example}
\label{P6:ex:5}
Let us consider the irregular time scale
$\mathbb{T}=\mathbb{P}_{1,1}=\bigcup\limits_{k=0}^{\infty}\left[2k,2k+1\right]$.
We show that for this time scale there is a difference between
the Euler--Lagrange equations \eqref{P6:eq:EL:nabla} and \eqref{P6:eq:EL:delta}.
The forward and backward jump operators  are given by
$$
\sigma(t)=
\begin{cases}
t,\quad t\in\bigcup\limits_{k=0}^{\infty}[2k,2k+1),\\
t+1, \quad t\in \bigcup\limits_{k=0}^{\infty}\left\{2k+1\right\},
\end{cases}
\quad
\rho(t)=
\begin{cases}
t,\quad t\in\bigcup\limits_{k=0}^{\infty}(2k,2k+1],\\
t-1, \quad t\in \bigcup\limits_{k=1}^{\infty}\left\{2k\right\},\\
0, \quad t = 0.
\end{cases}
$$
For $t = 0$ and $t\in \bigcup\limits_{k=0}^{\infty}\left(2k,2k+1\right)$,
equations \eqref{P6:eq:EL:nabla} and \eqref{P6:eq:EL:delta} coincide.
We can distinguish between them for
$t\in \bigcup\limits_{k=0}^{\infty}\left\{2k+1\right\}$
and $t\in \bigcup\limits_{k=1}^{\infty}\left\{2k\right\}$.
In what follows we use the notations \eqref{P6:eq:xi:chi}.
If $t\in \bigcup\limits_{k=0}^{\infty}\left\{2k+1\right\}$,
then we obtain from \eqref{P6:eq:EL:nabla} and \eqref{P6:eq:EL:delta}
the Euler--Lagrange equations
$\xi(t) + \chi(t) = c$ and $\xi(t) + \chi(t+1) = c$, respectively.
If $t\in \bigcup\limits_{k=1}^{\infty}\left\{2k\right\}$,
then the Euler--Lagrange equation \eqref{P6:eq:EL:nabla}
has the form $\xi(t-1) + \chi(t) = c$
while \eqref{P6:eq:EL:delta} takes the form $\xi(t) + \chi(t) = c$.
\end{example}


\subsection{Natural boundary conditions}
\label{sec:nbc}

In this section we minimize or maximize
the variational functional \eqref{P6:eq:problem},
but initial and/or terminal boundary condition $y(a)$ and/or $y(b)$ are not specified.
In what follows we obtain corresponding transversality conditions.

\begin{theorem}[Transversality condition at the initial time $t = a$]
\label{P6:th:trans:initial}
Let $\mathbb{T}$ be a time scale for which $\rho(\sigma(a))=a$.
If $\hat{y}$ is a local extremizer to \eqref{P6:eq:problem}
with $y(a)$ not specified, then
\begin{equation*}
\sum\limits_{i=1}^{k}H_{i}^{'}
\cdot
f_{iv}[\hat{y}](a)
+\sum\limits_{i=k+1}^{k+n}H_{i}^{'}\cdot
\left(
f_{iv}\lb\hat{y}\rb(\sigma(a))
- \int\limits^{\sigma(a)}_{a}f_{iy}\lb\hat{y}\rb(t)\nabla t
\right) = 0
\end{equation*}
holds together with the Euler--Lagrange equations
\eqref{P6:eq:EL:nabla} and \eqref{P6:eq:EL:delta}.
\end{theorem}

\begin{proof}
See \cite{MR3040923}.
\end{proof}

\begin{theorem}[Transversality condition at the terminal time $t = b$]
\label{P6:th:trans:terminal}
Let $\mathbb{T}$ be a time scale for which $\sigma(\rho(b))=b$.
If $\hat{y}$ is a local extremizer to \eqref{P6:eq:problem}
with $y(b)$ not specified, then
\begin{equation*}
\sum\limits_{i=1}^{k}H_{i}^{'}\cdot
\left(
f_{iv}[\hat{y}](\rho(b))
+
\int\limits_{\rho(b)}^{b}f_{iy}[\hat{y}](t)\Delta t
\right)
+\sum\limits_{i=k+1}^{k+n}H_{i}^{'}
\cdot
f_{iv}\lb\hat{y}\rb(b) = 0
\end{equation*}
holds together with the Euler--Lagrange equations
\eqref{P6:eq:EL:nabla} and \eqref{P6:eq:EL:delta}.
\end{theorem}

\begin{proof}
See \cite{MR3040923}.
\end{proof}

Several interesting results can be immediately obtained
from Theorems~\ref{P6:th:main},
\ref{P6:th:trans:initial} and \ref{P6:th:trans:terminal}.
An example of such results is given by Corollary~\ref{P6:cor:quotient}.

\begin{corollary}
\label{P6:cor:quotient}
If $\hat{y}$ is a solution to the problem
\begin{gather*}
\mathcal{L}[y]
=\frac{\int\limits_{a}^{b}f_{1}(t,y^{\sigma}(t),y^{\Delta}(t))
\Delta t}{\int\limits_{a}^{b}f_{2}(t,y^{\rho}(t),y^{\nabla}(t))\nabla t}
\longrightarrow \mathrm{extr},\\
(y(a)=y_{a}), \quad (y(b)=y_{b}),
\end{gather*}
then the Euler--Lagrange equations
$$
\frac{1}{\mathcal{F}_{2}}
\left(
f_{1v}[\hat{y}](\rho(t))-\int\limits_{a}^{\rho(t)}
f_{1y}[\hat{y}](\tau)\Delta \tau\right)
- \frac
{\mathcal{F}_{1}}
{\mathcal{F}_{2}^{2}}
\left(f_{2v}\lb\hat{y}\rb(t)-\int\limits_{a}^{t}
f_{2y}\lb\hat{y}\rb(\tau)\nabla \tau\right) = c,
$$
$t\in\T_{\kappa}$, and
$$
\frac{1}{\mathcal{F}_{2}}
\left(
f_{1v}[\hat{y}](t)-\int\limits_{a}^{t}
f_{1y}[\hat{y}](\tau)\Delta \tau\right)
-\frac{\mathcal{F}_{1}}{\mathcal{F}_{2}^{2}}
\left(f_{2v}\lb\hat{y}\rb(\sigma(t))
-\int\limits^{\sigma (t)}_{a}
f_{2y}\lb\hat{y}\rb(\tau)\nabla \tau\right)=c,
$$
$t\in\T^{\kappa}$, hold, where
$$
\mathcal{F}_{1}:={\int\limits_{a}^{b}
f_{1}(t,\hat{y}^{\sigma}(t),\hat{y}^{\Delta}(t))\Delta t}
\quad \text{ and } \quad
\mathcal{F}_{2}:={\int\limits_{a}^{b}
f_{2}(t,\hat{y}^{\rho}(t),\hat{y}^{\nabla}(t))\nabla t}.
$$
Moreover, if $y(a)$ is free and $\rho(\sigma(a))=a$, then
$$
\frac{1}{\mathcal{F}_{2}}
f_{1v}[\hat{y}](a)
-\frac{\mathcal{F}_{1}}{\mathcal{F}_{2}^{2}}
\left(f_{2v}\lb\hat{y}\rb(\sigma(a))-\int\limits_{a}^{\sigma(a)}
f_{2y}\lb\hat{y}\rb(t)\nabla t\right) =0;
$$
if $y(b)$ is free and $\sigma(\rho(b))=b$, then
$$
\frac{1}{\mathcal{F}_{2}}
\left(f_{1v}[\hat{y}](\rho(b))+\int\limits^{b}_{\rho(b)}
f_{1y}[\hat{y}](t)\Delta t\right)
-\frac{\mathcal{F}_{1}}{\mathcal{F}_{2}^{2}}
f_{2v}\lb\hat{y}\rb(b)=0.
$$
\end{corollary}


\subsection{Isoperimetric problems}
\label{sub:sec:iso:p}

Let us now consider the general delta--nabla composition
isoperimetric problem on time scales
subject to boundary conditions.
The problem consists of extremizing
\begin{multline}
\label{P6:eq:iso}
\mathcal L[y]=H\left(\int\limits_{a}^{b}f_{1}(t,y^{\sigma}(t),y^{\Delta}(t))\Delta t,
\ldots,\int\limits_{a}^{b}f_{k}(t,y^{\sigma}(t),y^{\Delta}(t))\Delta t,\right.\\
\left.\int\limits_{a}^{b}f_{k+1}(t,y^{\rho}(t),y^{\nabla}(t))\nabla t,\ldots,
\int\limits_{a}^{b}f_{k+n}(t,y^{\rho}(t),y^{\nabla}(t)) \nabla t \right)
\end{multline}
in the class of functions $y\in C^1_{k+m,n+p}$ satisfying given boundary conditions
\begin{equation}
\label{P6:eq:bound:conds:iso}
y(a)=y_{a},\quad y(b)=y_{b},
\end{equation}
and a generalized isoperimetric constraint
\begin{multline}
\label{P6:eq:iso:constraint}
\mathcal{K}[y]
=P\left(\int\limits_{a}^{b}g_{1}(t,y^{\sigma}(t),y^{\Delta}(t))\Delta t,
\ldots,\int\limits_{a}^{b}g_{m}(t,y^{\sigma}(t),y^{\Delta}(t))\Delta t,\right.\\
\left.\int\limits_{a}^{b}g_{m+1}(t,y^{\rho}(t),y^{\nabla}(t))\nabla t,\ldots,
\int\limits_{a}^{b}g_{m+p}(t,y^{\rho}(t),y^{\nabla}(t)) \nabla t \right)=d,
\end{multline}
where $y_{a},y_{b},d\in\mathbb{R}$. We assume that:
\begin{enumerate}
\item
the functions $H:\mathbb{R}^{n+k}\rightarrow\mathbb{R}$
and $P:\mathbb{R}^{m+p}\rightarrow\mathbb{R}$
have continuous partial derivatives with respect to all their arguments,
which we denote by $H_{i}^{'}$, $i=1,\ldots,n+k$,
and $P_{i}^{'}$, $i=1,\ldots,m+p$;
\item
functions $(t,y,v)\rightarrow f_{i}(t,y,v)$,
$i=1,\ldots, n+k$, and
$(t,y,v)\rightarrow g_{j}(t,y,v)$, $j=1,\ldots,m+p$,
from $[a,b]\times\mathbb{R}^{2}$ to $\mathbb{R}$,
have partial continuous derivatives with respect to $y$ and $v$ uniformly in
$t\in [a,b]$, which we denote by $f_{iy}$, $f_{iv}$, and $g_{jy}, g_{jv}$;
\item for all $y\in C_{k+m,n+p}^{1}$,
$f_{i}$, $f_{iy}$, $f_{iv}$ and $g_{j},g_{jy}$, $g_{jv}$
are rd-continuous in $t\in [a,b]^{\kappa}$,
$i=1,\ldots,k$, $j=1,\ldots,m$,
and ld-continuous in $t\in [a,b]_{\kappa}$,
$i=k+1,\ldots,k+n$, $j=m+1,\ldots,m+p$.
\end{enumerate}
A function $y\in C^{1}_{k+m, n+p}$ is said to be admissible provided
it satisfies the boundary conditions \eqref{P6:eq:bound:conds:iso}
and the isoperimetric constraint \eqref{P6:eq:iso:constraint}.
For brevity, we omit the argument of $P_{i}^{'}$:
$P_{i}^{'}:=\frac{\partial P}{\partial \mathcal{G}_{i}}(\mathcal{G}_{1}(\hat{y}),
\ldots,\mathcal{G}_{m+p}(\hat{y}))$ for $i=1,\ldots,m+p$,
with
$$
\mathcal{G}_{i}(\hat{y})=\int\limits_{a}^{b}
g_{i}(t,\hat{y}^{\sigma}(t),\hat{y}^{\Delta}(t))\Delta t,
\quad i=1,\ldots,m,
$$
and
$$
\mathcal{G}_{i}(\hat{y})=\int\limits_{a}^{b}
g_{i}(t,\hat{y}^{\rho}(t),\hat{y}^{\nabla}(t))\nabla t,
\quad i=m+1,\ldots,m+p.
$$

\begin{definition}
We say that an admissible function $\hat{y}$ is a local minimizer
(respectively, a local maximizer) to the isoperimetric problem
\eqref{P6:eq:iso}--\eqref{P6:eq:iso:constraint}, if there exists a $\delta >0$
such that $\mathcal{L}[\hat{y}]\leqslant \mathcal{L}[y]$
(respectively, $\mathcal{L}[\hat{y}]\geqslant \mathcal{L}[y]$)
for all admissible functions $y\in C_{k+m,n+p}^{1}$
satisfying the inequality $||y-\hat{y}||_{1,\infty}<\delta$.
\end{definition}

Let us define $u$ and $w$ by
\begin{equation}
\label{P6:eq:u:w}
\begin{gathered}
u(t):=
\sum\limits_{i=1}^{m}P_{i}^{'}\cdot
\left(g_{iv}[\hat{y}](t)-\int\limits_{a}^{t}
g_{iy}[\hat{y}](\tau)\Delta \tau \right),\\
w(t):=
\sum\limits_{i=m+1}^{m+p}P_{i}^{'}\cdot
\left(g_{iv}\lb\hat{y}\rb(t)-\int\limits_{a}^{t}
g_{iy}\lb\hat{y}\rb(\tau)\nabla \tau \right).
\end{gathered}
\end{equation}

\begin{definition}
An admissible function $\hat{y}$ is said to be an extremal for $\mathcal{K}$ if
$u(t) + w(\sigma(t)) = const$ and $u(\rho(t)) + w(t) = const$ for all $t\in[a,b]_\kappa^\kappa$.
An extremizer (i.e., a local minimizer or a local maximizer)
to problem \eqref{P6:eq:iso}--\eqref{P6:eq:iso:constraint} that is not an extremal for $\mathcal{K}$
is said to be a normal extremizer; otherwise (i.e., if it is an extremal for $\mathcal{K}$),
the extremizer is said to be abnormal.
\end{definition}

\begin{theorem}[Optimality condition to the
isoperimetric problem \eqref{P6:eq:iso}--\eqref{P6:eq:iso:constraint}]
\label{P6:th:conds:iso}
Let $\chi$ and $\xi$ be given as in \eqref{P6:eq:xi:chi}, and
$u$ and $w$ be given as in \eqref{P6:eq:u:w}.
If $\hat{y}$ is a normal extremizer to the isoperimetric problem
\eqref{P6:eq:iso}--\eqref{P6:eq:iso:constraint}, then there exists
a real number $\lambda$ such that
\begin{enumerate}
\item $\xi^{\rho}(t)+\chi(t)-\lambda\left(u^{\rho}(t)+w(t)\right)= const$;
\item $\xi(t)+\chi^{\sigma}(t)-\lambda\left(u^{\rho}(t)+w(t)\right)= const$;
\item $\xi^{\rho}(t)+\chi(t)-\lambda\left(u(t)+w^{\sigma}(t)\right)= const$;
\item $\xi(t)+\chi^{\sigma}(t)-\lambda\left(u(t)+w^{\sigma}(t)\right)= const$;
\end{enumerate}
for all $t\in [a,b]^{\kappa}_{\kappa}$.
\end{theorem}

\begin{proof}
See proof of Theorem~3.9 in \cite{MR3040923}.
\end{proof}


\subsection{Illustrative examples}
\label{sec:examples}

In this section we consider three examples which illustrate
the results of Theorem~\ref{P6:th:main} and Theorem~\ref{P6:th:conds:iso}.
We begin with a nonautonomous problem.

\begin{example}
\label{P6:ex:1}
Consider the problem
\begin{equation}
\label{P6:eq:9}
\begin{gathered}
\mathcal{L}[y]=
\frac{\int\limits_{0}^{1} t y^{\Delta}(t) \Delta t}
{\int\limits_{0}^{1}(y^{\nabla}(t))^{2}\nabla t}
\longrightarrow \min, \\
y(0)=0, \quad y(1)=1.
\end{gathered}
\end{equation}
If $y$ is a local minimizer to problem \eqref{P6:eq:9},
then the Euler--Lagrange equations of Corollary~\ref{P6:cor:quotient} must hold, i.e.,
$$
\frac{1}{\mathcal{F}_{2}}\rho(t)-2\frac{\mathcal{F}_{1}}{\mathcal{F}_{2}^{2}}
y^{\nabla}(t)=c, \quad t \in \mathbb{T}_{\kappa},
$$
and
$$
\frac{1}{\mathcal{F}_{2}}t-2\frac{\mathcal{F}_{1}}{\mathcal{F}_{2}^{2}}
y^{\nabla}(\sigma(t))=c, \quad t \in \mathbb{T}^{\kappa},
$$
where
$\mathcal{F}_{1}:=\mathcal{F}_{1}(y)=\int\limits_{0}^{1}t y^{\Delta}(t)\Delta t$
and $\mathcal{F}_{2}:=\mathcal{F}_{2}(y)=\int\limits_{0}^{1}(y^{\nabla}(t))^{2}\nabla t$.
Let us consider the second equation. Using \eqref{eq:delta:nabla:sigma}
of Theorem~\ref{th:differ:delta:nabla}, it can be written as
\begin{equation}
\label{P6:eq:10}
\frac{1}{\mathcal{F}_{2}}t-2\frac{\mathcal{F}_{1}}{\mathcal{F}_{2}^{2}}y^{\Delta}(t)=c,
\quad t \in \mathbb{T}^{\kappa}.
\end{equation}
Solving \eqref{P6:eq:10} subject to the boundary conditions $y(0)=0$ and $y(1)=1$ gives
\begin{equation}
\label{P6:eq:ex:1:sol}
y(t)=
\frac{1}{2Q}\int\limits_{0}^{t}\tau\Delta\tau
-t\left(\frac{1}{2Q}\int\limits_{0}^{1}\tau\Delta\tau -1\right),
\quad t \in \mathbb{T}^{\kappa},
\end{equation}
where $Q:=\frac{\mathcal{F}_{1}}{\mathcal{F}_{2}}$.
Therefore, the solution depends on the time scale.
Let us consider two examples: $\mathbb{T}=\mathbb{R}$
and $\mathbb{T}=\left\{0,\frac{1}{2},1\right\}$.
On $\mathbb{T}=\mathbb{R}$, from \eqref{P6:eq:ex:1:sol}
we obtain
\begin{equation}
\label{P6:eq:11}
y(t)=\frac{1}{4Q}t^{2}+\frac{4Q-1}{4Q}t,
\quad\quad y^{\Delta}(t) = y^{\nabla}(t) = y'(t)=\frac{1}{2Q}t+\frac{4Q-1}{4Q}
\end{equation}
as solution of \eqref{P6:eq:10}. Substituting \eqref{P6:eq:11} into
$\mathcal{F}_{1}$ and $\mathcal{F}_{2}$ gives
$\mathcal{F}_{1}=\frac{12Q+1}{24Q}$ and
$\mathcal{F}_{2}=\frac{48Q^{2}+1}{48Q^{2}}$, that is,
\begin{equation}
\label{P6:eq:Q:1}
Q=\frac{2Q(12Q+1)}{48Q^{2}+1}.
\end{equation}
Solving equation \eqref{P6:eq:Q:1} we get
$Q\in\left\{\frac{3-2\sqrt{3}}{12},\frac{3+2\sqrt{3}}{12}\right\}$.
Because \eqref{P6:eq:9} is a minimizing problem,
we select $Q=\frac{3-2\sqrt{3}}{12}$ and we get the extremal
\begin{equation}
\label{P6:eq:12}
y(t)=-(3+2\sqrt{3}) t^{2} + (4 + 2 \sqrt{3}) t.
\end{equation}
If $\mathbb{T}=\left\{0,\frac{1}{2},1\right\}$,
then from \eqref{P6:eq:ex:1:sol} we obtain
$y(t)=\frac{1}{8Q}\sum\limits_{k=0}^{2t-1}k+\frac{8Q-1}{8Q}t$, that is,
\begin{equation*}
y(t)=
\begin{cases}
0, & \text{ if } t=0,\\
\frac{8Q-1}{16Q}, & \text{ if } t=\frac{1}{2},\\
1, & \text{ if } t=1.
\end{cases}
\end{equation*}
Direct calculations show that
\begin{equation}
\label{P6:eq:13}
\begin{gathered}
y^{\Delta}(0)=\frac{y(\frac{1}{2})-y(0)}{\frac{1}{2}}=\frac{8Q-1}{8Q},
\quad y^{\Delta}\left(\frac{1}{2}\right)
=\frac{y(1)-y(\frac{1}{2})}{\frac{1}{2}}=\frac{8Q+1}{8Q},\\
y^{\nabla}\left(\frac{1}{2}\right)
=\frac{y(\frac{1}{2})-y(0)}{\frac{1}{2}}=\frac{8Q-1}{8Q},
\quad y^{\nabla}(1)=\frac{y(1)-y(\frac{1}{2})}{\frac{1}{2}}
=\frac{8Q+1}{8Q}.
\end{gathered}
\end{equation}
Substituting \eqref{P6:eq:13} into the integrals
$\mathcal{F}_{1}$ and $\mathcal{F}_{2}$ gives
\begin{equation*}
\mathcal{F}_{1}=\frac{8Q+1}{32Q},
\quad
\mathcal{F}_{2}=\frac{64Q^{2}+1}{64Q^{2}},
\quad
Q=\frac{\mathcal{F}_{1}}{\mathcal{F}_{2}}=\frac{2Q(8Q+1)}{64Q^{2}+1}.
\end{equation*}
Thus, we obtain the equation $64Q^{2}-16Q-1=0$.
The solutions to this equation are:
$Q\in \left\{\frac{1-\sqrt{2}}{8}, \frac{1+\sqrt{2}}{8}\right\}$.
We are interested in the minimum value $Q$, so we select
$Q = \frac{1+\sqrt{2}}{8}$ to get the extremal
\begin{equation}
\label{P6:eq:14}
y(t)
=\begin{cases}
0, & \hbox{ if } t=0,\\
1-\frac{\sqrt{2}}{2}, & \hbox{ if } t=\frac{1}{2},\\
1, &  \hbox{ if }t=1.
\end{cases}
\end{equation}
Note that the extremals \eqref{P6:eq:12} and \eqref{P6:eq:14} are different:
for \eqref{P6:eq:12} one has $x(1/2) = \frac{5}{4} + \frac{\sqrt{3}}{2}$.
\end{example}

In the previous example, the variational functional is given
by the ratio of a delta and a nabla integral. Now we discuss
a variational problem where the composition is expressed by
the product of three time-scale integrals.

\begin{example}
\label{P6:ex:3}
Consider the problem
\begin{equation}
\label{P6:eq:18}
\begin{gathered}
\mathcal{L}[y]=
\left(\int\limits_{0}^{3} t y^{\Delta}(t) \Delta t\right)
\left(\int\limits_{0}^{3} y^{\Delta}(t)\left(1+t\right)\Delta t\right)
\left(\int\limits_{0}^{3}\left[\left(y^{\nabla}(t)\right)^{2}+t\right]\nabla t\right)
\longrightarrow \min,\\
y(0)=0,\quad y(3)=3.
\end{gathered}
\end{equation}
If $y$ is a local minimizer to problem \eqref{P6:eq:18},
then the Euler--Lagrange equations must hold, and we can write that
\begin{equation}
\label{P6:eq:19}
\left(\mathcal{F}_{1}\mathcal{F}_{3}+\mathcal{F}_{2}\mathcal{F}_{3}\right)t
+\mathcal{F}_{1}\mathcal{F}_{3}+2\mathcal{F}_{1}\mathcal{F}_{2}
y^{\nabla}(\sigma(t))=c, \quad t \in \mathbb{T}^{\kappa},
\end{equation}
where $c$ is a constant,
$\mathcal{F}_{1}:=\mathcal{F}_{1}(y)
=\int\limits_{0}^{3} t y^{\Delta}(t) \Delta t$, $\mathcal{F}_{2}:=\mathcal{F}_{2}(y)
=\int\limits_{0}^{3} y^{\Delta}(t)\left(1+t\right)\Delta t$,
and $\mathcal{F}_{3}:=\mathcal{F}_{3}(y)
=\int\limits_{0}^{3}\left[\left(y^{\nabla}(t)\right)^{2}+t\right]\nabla t$.
Using relation \eqref{eq:delta:nabla:sigma}, we can write \eqref{P6:eq:19} as
\begin{equation}
\label{P6:eq:20}
\left(\mathcal{F}_{1}\mathcal{F}_{3}+\mathcal{F}_{2}\mathcal{F}_{3}\right)t
+\mathcal{F}_{1}\mathcal{F}_{3}+2\mathcal{F}_{1}\mathcal{F}_{2}y^{\Delta}(t)=c,
\quad t \in \mathbb{T}^{\kappa}.
\end{equation}
Using the boundary conditions $y(0)=0$ and $y(3)=3$,
from \eqref{P6:eq:20} we get that
\begin{equation}
\label{P6:eq:ex:3:sol}
y(t)=\left(1+\frac{Q}{3}\int\limits_{0}^{3}\tau\Delta \tau\right) t
-Q \int\limits_{0}^{t}\tau\Delta \tau,
\quad t \in \mathbb{T}^{\kappa},
\end{equation}
where $Q=\frac{\mathcal{F}_{1}\mathcal{F}_{3}
+\mathcal{F}_{2}\mathcal{F}_{3}}{2\mathcal{F}_{1}\mathcal{F}_{2}}$.
Therefore, the solution depends on the time scale.
Let us consider $\mathbb{T}=\mathbb{R}$ and $\mathbb{T}
=\left\{0,\frac{1}{2},1,\frac{3}{2},2,\frac{5}{2},3\right\}$.
On $\mathbb{T}=\mathbb{R}$, expression \eqref{P6:eq:ex:3:sol} gives
\begin{equation}
\label{P6:eq:21}
y(t)=\left(\frac{2+3Q}{2}\right) t - \frac{Q}{2}t^{2},
\quad y^{\Delta}(t)= y^{\nabla}(t) = y'(t) = \frac{2+3Q}{2}-Qt
\end{equation}
as solution of \eqref{P6:eq:20}. Substituting \eqref{P6:eq:21}
into $\mathcal{F}_{1}$, $\mathcal{F}_{2}$ and $\mathcal{F}_{3}$ gives:
$$
\mathcal{F}_{1}=\frac{18-9Q}{4},
\quad
\mathcal{F}_{2}=\frac{30-9Q}{4},
\quad
\mathcal{F}_{3}=\frac{9Q^{2}+30}{4}.
$$
Solving equation $9Q^{3} - 36 Q^{2} + 45 Q - 40=0$, one finds
the solution
$$
Q = \frac{1}{27}\left[ 36+ \sqrt[3]{24786-729\sqrt{1155}}
+9\sqrt[3]{34+\sqrt{1155}}\right]\approx 2,7755
$$
and the extremal $y(t)=5,16325t-1,38775t^{2}$.

Let us consider now the time scale $\mathbb{T}
=\left\{0,\frac{1}{2},1,\frac{3}{2},2,\frac{5}{2},3\right\}$.
From \eqref{P6:eq:ex:3:sol}, we obtain
\begin{equation}
\label{P6:eq:22}
y(t)=\left(\frac{4+5Q}{4}\right)t-\frac{Q}{4}
\sum\limits_{k=0}^{2t-1}k
=
\begin{cases}
0, & \hbox{ if } t=0,\\
\frac{4+5Q}{8}, & \hbox{ if } t=\frac{1}{2},\\
1+Q, &  \hbox{ if } t=1,\\
\frac{12+9Q}{8}, & \hbox{ if } t=\frac{3}{2},\\
2+Q, & \hbox{ if } t=2,\\
\frac{20+5Q}{8}, & \hbox{ if } t=\frac{5}{2},\\
3, & \hbox{ if } t=3
\end{cases}
\end{equation}
as solution of \eqref{P6:eq:20}. Substituting \eqref{P6:eq:22}
into $\mathcal{F}_{1}$, $\mathcal{F}_{2}$
and $\mathcal{F}_{3}$, yields
\begin{equation*}
\mathcal{F}_{1}=\frac{60-35Q}{16},
\quad \mathcal{F}_{2}=\frac{108-35Q}{16},
\quad \mathcal{F}_{3}=\frac{35Q^{2}+132}{16}.
\end{equation*}
Solving equation $245Q^{3}-882Q^{2}+1110-\frac{5544}{5}=0$,
we get $Q\approx 2,5139$ and the extremal
\begin{equation}
\label{eq:extremal:50}
y(t)=
\begin{cases}
0, & \hbox{ if } t=0,\\
2,0711875, & \hbox{ if } t=\frac{1}{2},\\
3,5139, &  \hbox{ if } t=1,\\
4,3281375, & \hbox{ if } t=\frac{3}{2},\\
4,5139, & \hbox{ if } t=2,\\
4,0711875, & \hbox{ if } t=\frac{5}{2},\\
3, & \hbox{ if } t=3
\end{cases}
\end{equation}
for problem \eqref{P6:eq:18} on
$\mathbb{T}=\left\{0,\frac{1}{2},1,\frac{3}{2},2,\frac{5}{2},3\right\}$.

In order to illustrate the difference between composition
of mixed delta-nabla integrals and pure delta or nabla situations,
we consider now two variants of problem \eqref{P6:eq:18}:
(i) the first consisting of delta operators only:
\begin{equation}
\label{survey:1}
\begin{gathered}
\mathcal{L}[y]=
\left(\int\limits_{0}^{3} t y^{\Delta}(t) \Delta t\right)
\left(\int\limits_{0}^{3} y^{\Delta}(t)\left(1+t\right)\Delta t\right)
\left(\int\limits_{0}^{3}\left[\left(y^{\Delta}(t)\right)^{2}+t\right]\Delta t\right)
\longrightarrow \min;
\end{gathered}
\end{equation}
(ii) the second of nabla operators only:
\begin{equation}
\label{survey:2}
\begin{gathered}
\mathcal{L}[y]=
\left(\int\limits_{0}^{3} t y^{\nabla}(t) \nabla t\right)
\left(\int\limits_{0}^{3} y^{\nabla}(t)\left(1+t\right)\nabla t\right)
\left(\int\limits_{0}^{3}\left[\left(y^{\nabla}(t)\right)^{2}+t\right]\nabla t\right)
\longrightarrow \min.
\end{gathered}
\end{equation}
Both problems (i) and (ii) are subject to the same boundary
conditions as in \eqref{P6:eq:18}:
\begin{equation}
\label{eq:boundaryConditions}
y(0)=0,\quad y(3)=3.
\end{equation}
All three problems \eqref{P6:eq:18}, \eqref{survey:1} and \eqref{eq:boundaryConditions},
and \eqref{survey:2}--\eqref{eq:boundaryConditions},
coincide in $\mathbb{R}$. Consider, as before, the time scale
$\mathbb{T}=\left\{0,\frac{1}{2},1,\frac{3}{2},2,\frac{5}{2},3\right\}$.
Recall that problem \eqref{P6:eq:18} has extremal \eqref{eq:extremal:50}.
(i) Now, let us consider the delta problem \eqref{survey:1} and \eqref{eq:boundaryConditions}.
We obtain
\begin{equation*}
\mathcal{F}_{1}=\frac{60-35Q}{16},
\quad \mathcal{F}_{2}=\frac{108-35Q}{16},
\quad \mathcal{F}_{3}=\frac{35Q^{2}+108}{16}
\end{equation*}
and the equation $245Q^{3}-882Q^{2}+1026-\frac{5436}{5}=0$.
Its numerical solution $Q\approx 2,5216$ entails the extremal
\begin{equation*}
y(t)=
\begin{cases}
0, & \hbox{ if } t=0,\\
2,076, & \hbox{ if } t=\frac{1}{2},\\
3,5216, &  \hbox{ if } t=1,\\
4,3368, & \hbox{ if } t=\frac{3}{2},\\
4,5216, & \hbox{ if } t=2,\\
4,076, & \hbox{ if } t=\frac{5}{2},\\
3, & \hbox{ if } t=3.
\end{cases}
\end{equation*}
(ii) In the latter nabla problem \eqref{survey:2}--\eqref{eq:boundaryConditions} we have
\begin{equation*}
\mathcal{F}_{1}=\frac{84-35Q}{16},
\quad \mathcal{F}_{2}=\frac{132-35Q}{16},
\quad \mathcal{F}_{3}=\frac{35Q^{2}+132}{16}
\end{equation*}
and the equation $175Q^{3}-810Q^{2}+1122-\frac{7128}{7}=0$.
Using its numerical solution $Q\approx 3,1097$ we get the extremal
\begin{equation*}
y(t)=
\begin{cases}
0, & \hbox{ if } t=0,\\
2,4942, & \hbox{ if } t=\frac{1}{2},\\
4,1907, & \hbox{ if } t=1,\\
5,0895, & \hbox{ if } t=\frac{3}{2},\\
5,1907, & \hbox{ if } t=2,\\
4,4942, & \hbox{ if } t=\frac{5}{2},\\
3, & \hbox{ if } t=3.
\end{cases}
\end{equation*}
\end{example}

Finally, we apply the results of Section~\ref{sub:sec:iso:p}
to an isoperimetric problem.

\begin{example}
\label{P6:ex:4}
Let us consider the problem of extremizing
\begin{equation*}
\mathcal{L}[y]=
\frac{
\int\limits_{0}^{1}(y^{\Delta}(t))^{2}\Delta t}
{\int\limits_{0}^{1} ty^{\nabla}(t)\nabla t}
\end{equation*}
subject to the boundary conditions
$y(0)=0$ and $y(1)=1$ and the isoperimetric constraint
$$
\mathcal{K}[y]=\int\limits_{0}^{1} ty^{\nabla}(t)\nabla t=1.
$$
Applying Theorem~\ref{P6:th:conds:iso}, we get the nabla differential equation
\begin{equation}
\label{P6:eq:23}
\frac{2}{\mathcal{F}_{2}}y^{\nabla}(t)
- \left(\lambda + \frac{\mathcal{F}_{1}}{(\mathcal{F}_{2})^{2}}\right) t = c,
\quad t \in \mathbb{T}^{\kappa}_{\kappa}.
\end{equation}
Solving this equation, we obtain
\begin{equation}
\label{P6:eq:ex:4:sol}
y(t)=\left(1-Q\int\limits_{0}^{1}\tau\nabla\tau\right)t
+Q\int\limits_{0}^{t}\tau\nabla \tau,
\end{equation}
where $Q=\frac{\mathcal{F}_{2}}{2}\left(\frac{\mathcal{F}_{1}}{(\mathcal{F}_{2})^{2}}+\lambda\right)$.
Therefore, the solution of equation \eqref{P6:eq:23} depends on the time scale.
Let us consider $\mathbb{T}=\mathbb{R}$ and $\mathbb{T}=\left\{0,\frac{1}{2},1\right\}$.

On $\mathbb{T}=\mathbb{R}$, from \eqref{P6:eq:ex:4:sol} we obtain that
$y(t)=\frac{2-Q}{2}t+\frac{Q}{2}t^{2}$.
Substituting this expression for $y$ into the integrals
$\mathcal{F}_{1}$ and $\mathcal{F}_{2}$,
gives $\mathcal{F}_{1}=\frac{Q^{2}+12}{12}$ and $\mathcal{F}_{2}=\frac{Q+6}{12}$.
Using the given isoperimetric constraint,
we obtain $Q=6$, $\lambda =8$, and
$y(t)=3t^{2}-2t$.
Let us consider now the time scale $\mathbb{T}=\left\{0,\frac{1}{2},1\right\}$.
From \eqref{P6:eq:ex:4:sol}, we have
$$
y(t)=\frac{4-3Q}{4}t+Q\sum\limits_{k=1}^{2t}\frac{k}{4}
=
\begin{cases}
0, & \hbox{ if } t=0,\\
\frac{4-Q}{8}, & \hbox{ if } t=\frac{1}{2},\\
1, & \hbox{ if } t=1.
\end{cases}
$$
Simple calculations show that
\begin{equation*}
\begin{split}
&
\mathcal{F}_{1}=\sum\limits_{k=0}^{1}\frac{1}{2} \left(y^{\Delta}\left(\frac{k}{2}\right)\right)^{2}
=\frac{1}{2}\left(y^{\Delta}(0)\right)^{2}
+\frac{1}{2}\left(y^{\Delta}\left(\frac{1}{2}\right)\right)^{2}=\frac{Q^{2}+16}{16},
\\
&
\mathcal{F}_{2}=\sum\limits_{k=1}^{2}\frac{1}{4}k y^{\nabla}\left(\frac{k}{2}\right)
=\frac{1}{4} y^{\nabla}\left(\frac{1}{2}\right)+\frac{1}{2}y^{\nabla}(1)=\frac{Q+12}{16}
\end{split}
\end{equation*}
and $\mathcal{K}(y)=\frac{Q+12}{16}=1$. Therefore, $Q=4$, $\lambda=6$, and we have the extremal
$$
y(t)=
\begin{cases}
0, &  \hbox{ if } t \in \left\{0, \frac{1}{2}\right\},\\
1, &  \hbox{ if } t=1.
\end{cases}
$$
\end{example}


\section{Conclusions}
\label{sec:conc}

In this survey we collected some of our recent research
on direct and inverse problems of the calculus of variations
on arbitrary time scales. For infinity horizon variational problems
on time scales we refer the reader to \cite{MyID:254,MR2994055}.
We started by studying inverse problems of the calculus of variations,
which have not been studied before in the time-scale framework.
First we derived a general form of a variational functional
which attains a local minimum at a given function $y_{0}$ under
Euler--Lagrange and strengthened Legendre conditions (Theorem~\ref{P3:th:integrand:2}).
Next we considered a new approach to the inverse problem of the calculus of variations
by using an integral perspective instead of the classical differential point of view.
In order to solve the problem, we introduced new definitions: (i) self-adjointness
of an integro-differential equation, and (ii) equation of variation.
We obtained a necessary condition for an integro-differential equation
to be an Euler--Lagrange equation on an arbitrary time scale $\mathbb{T}$
(Theorem~\ref{P4:th:necess:EL:int:diff}). It remains open the question of sufficiency.
Finally, we developed the direct calculus of variations
by considering functionals that are a composition of a certain scalar function
with delta and nabla integrals of a vector valued field. For such problems we obtained
delta-nabla Euler--Lagrange equations in integral form (Theorem~\ref{P6:th:main}),
transversality conditions (Theorems~\ref{P6:th:trans:initial} and \ref{P6:th:trans:terminal})
and necessary optimality conditions for isoperimetric problems (Theorem~\ref{P6:th:conds:iso}).
To consider such general mixed delta-nabla variational problems on unbounded time scales
(infinite horizon problems) remains also an open direction of research. Another interesting
open research direction consists to study delta-nabla inverse problems of calculus
of variations for composition functionals and their conservation laws \cite{MR2098297}.


\section*{Acknowledgments}

This work was partially supported by Portuguese funds through the
\emph{Center for Research and Development in Mathematics and Applications} (CIDMA),
and \emph{The Portuguese Foundation for Science and Technology} (FCT), within
project UID/MAT/ 04106/2013. The authors are grateful to two anonymous
referees for their valuable comments and suggestions.




\begin{thebibliography}{99}

\bibitem{MR1908827}
Ahlbrandt, C. D., Morian, C.:
Partial differential equations on time scales.
J. Comput. Appl. Math.
\textbf{141}, no.~1-2, 35--55 (2002)

\bibitem{MR1770981}
Albu, I. D., Opri\c s, D.:
Helmholtz type condition for mechanical integrators.
Novi Sad J. Math. \textbf{29}, no.~3, 11--21 (1999)

\bibitem{livro:FC:Comput}
Almeida, R., Pooseh, S., Torres, D. F. M.:
Computational methods in the fractional calculus of variations.
Imperial College Press, London (2015).

\bibitem{Almeida:iso}
Almeida, R., Torres, D. F. M.:
Isoperimetric problems on time scales with nabla derivatives.
J. Vib. Control
\textbf{15}, no.~6, 951--958 (2009)
{\tt arXiv:0811.3650}

\bibitem{MR2218315}
Atici, F. M. , Biles, D. C., Lebedinsky, A.:
An application of time scales to economics.
Math. Comput. Modelling
\textbf{43}, no.~7-8, 718--726 (2006)

\bibitem{AticiGreen's_functions}
Atici, F.M., Guseinov, G. Sh.: .
On Green's functions and positive solutions for boundary value problems on time scales
J. Comput. Appl. Math.
\textbf{141}, no.~1-2, 75--99 (2002)

\bibitem{Atici:comparison}
Atici, F. M. , McMahan, C. S.:
A comparison in the theory of calculus of variations
on time scales with an application to the Ramsey model.
Nonlinear Dyn. Syst. Theory
\textbf{9}, no.~1, 1--10 (2009)

\bibitem{Atici:HMMS}
Atici, F. M., Uysal, F.:
A production-inventory model of HMMS on time scales.
Appl. Math. Lett.
\textbf{21}, no.~3, 236--243 (2008)

\bibitem{Bartos1}
Bartosiewicz, Z., Kotta, \"{U}., Paw{\l}uszewicz, E., Wyrwas, M.:
Control systems on regular time scales and their differential rings.
Math. Control Signals Systems
\textbf{22}, no.~3, 185--201 (2011)

\bibitem{MyID:179}
Bastos, N. R. O., Ferreira, R. A. C., Torres, D. F. M.:
Discrete-time fractional variational problems.
Signal Process.
\textbf{91}, no.~3, 513--524 (2011)
{\tt arXiv:1005.0252}

\bibitem{BohnerCOVOTS}
Bohner, M.:
Calculus of variations on time scales
Dynam. Systems Appl.
\textbf{13}, no.~3-4, 339--349 (2004)

\bibitem{MR2662835}
Bohner, M. J., Ferreira, R. A. C., Torres, D. F. M.:
Integral inequalities and their applications to the calculus of variations on time scales.
Math. Inequal. Appl.
\textbf{13}, no.~3, 511--522 (2010)
{\tt arXiv:1001.3762}

\bibitem{BohGus1}
Bohner, M., Guseinov, G. Sh.:
Partial differentiation on time scales.
Dynam. Systems Appl.
\textbf{13}, no.~3-4, 351--379 (2004)

\bibitem{mb:gg:ap}
Bohner, M., Guseinov, G., Peterson, A.:
Introduction to the time scales calculus.
In: Advances in dynamic equations on time scales, pp.~1--15.
Birkh\"auser Boston, Boston, MA (2003)

\bibitem{BohnerDEOTS}
Bohner, M., Peterson, A.:
Dynamic equations on time scales.
Birkh\"auser Boston, Boston, MA (2001)

\bibitem{MBbook2003}
Bohner, M., Peterson, A.:
Advances in dynamic equations on time scales.
Birkh\"auser Boston, Boston, MA (2003)

\bibitem{Helmholtz}
Bourdin, L., Cresson, J.:
Helmholtz's inverse problem of the discrete calculus of variations.
J. Difference Equ. Appl.
\textbf{19}, no.~9, 1417--1436 (2013)

\bibitem{cc:dual}
Caputo, M. C.:
Time scales: from nabla calculus to delta calculus and vice versa via duality.
Int. J. Difference Equ.
\textbf{5}, no.~1, 25--40 (2010)

\bibitem{MR1609049}
Cr\u aciun, D., Opri\c s, D.:
The Helmholtz conditions for the difference equations systems.
Balkan J. Geom. Appl. \textbf{1}, no.~2, 21--30 (1996)

\bibitem{Davis}
Davis, D. R.:
The inverse problem of the calculus of variations in higher space.
Trans. Amer. Math. Soc.
\textbf{30}, no.~4, 710--736 (1928)

\bibitem{Douglas}
Douglas, J.:
Solution of the inverse problem of the calculus of variations.
Trans. Amer. Math. Soc.
\textbf{50}, 71--128 (1941)

\bibitem{PhD:thesis:Monia}
Dryl, M.:
Calculus of variations on time scales and applications to economics.
PhD Thesis, University of Aveiro (2014)

\bibitem{MyID:267}
Dryl, M., Malinowska, A. B., Torres, D. F. M.:
A time-scale variational approach to inflation, unemployment and social loss.
Control Cybernet.
\textbf{42}, no.~2, 399--418 (2013)
{\tt arXiv:1304.5269}

\bibitem{Dryl:Torres:1}
Dryl, M., Malinowska, A. B., Torres, D. F. M.:
An inverse problem of the calculus of variations on arbitrary time scales.
Int. J. Difference Equ.
\textbf{9}, no.~1, 53--66 (2014)
{\tt arXiv:1401.8232}

\bibitem{MyID:254}
Dryl, M., Torres, D. F. M.:
Necessary optimality conditions for infinite horizon variational problems on time scales.
Numer. Algebra Control Optim.
\textbf{3}, no.~1, 145--160 (2013)
{\tt arXiv:1212.0988}

\bibitem{MR3040923}
Dryl, M., Torres, D. F. M.:
The delta-nabla calculus of variations for composition functionals on time scales.
Int. J. Difference Equ.
\textbf{8}, no.~1, 27--47 (2013)
{\tt arXiv:1211.4368}

\bibitem{MyID:291}
Dryl, M., Torres, D. F. M.:
Necessary condition for an Euler-Lagrange equation on time scales.
Abstr. Appl. Anal.
\textbf{7}, Art. ID 631281 (2014)
{\tt arXiv:1403.3252}

\bibitem{Ernst}
Ernst, T.:
The different tongues of $q$-calculus.
Proc. Est. Acad. Sci.
\textbf{57}, no.~2, 81--99 (2008)

\bibitem{FerreiraTorres}
Ferreira, R. A. C., Torres, D. F. M.:
Necessary optimality conditions for the calculus of variations on time scales (2007) 
{\tt arXiv:0704.0656} 

\bibitem{MR2405376}
Ferreira, R. A. C., Torres, D. F. M.:
Remarks on the calculus of variations on time scales.
Int. J. Ecol. Econ. Stat.
\textbf{9}, no.~F07, 65--73 (2007)
{\tt arXiv:0706.3152}

\bibitem{MR2668257}
Ferreira, R. A. C., Torres, D. F. M.:
Isoperimetric problems of the calculus of variations on time scales.
In: Nonlinear analysis and optimization II. Optimization. 123--131,
Contemp. Math., 514, Amer. Math. Soc., Providence, RI (2010)
{\tt arXiv:0805.0278}

\bibitem{MR2879335}
Girejko, E., Malinowska, A. B., Torres, D. F. M.:
The contingent epiderivative and the calculus of variations on time scales.
Optimization
\textbf{61}, no.~3, 251--264 (2012)
{\tt arXiv:1007.0509}

\bibitem{MR2957726}
Girejko, E., Torres, D. F. M.:
The existence of solutions for dynamic inclusions on time scales via duality.
Appl. Math. Lett.
\textbf{25} no.~11, 1632--1637 (2012)
{\tt arXiv:1201.4495}

\bibitem{phdHilger}
Hilger, S.:
Ein ma{\ss}kettenkalk\"{u}l mit anwendung auf zentrumsmannigfaltigkeiten.
PhD Thesis, Universit\"{a}t W\"{u}rzburg (1988)

\bibitem{Hilger}
Hilger, S.:
Analysis on measure chains---a unified approach to continuous and discrete calculus.
Results Math.
\textbf{18}, no.~1-2, 18--56 (1990)

\bibitem{Hilger2}
Hilger, S.:
Differential and difference calculus---unified!.
Nonlinear Anal.
\textbf{30} no.~5, 2683--2694 (1997)

\bibitem{HilscgerZeidan}
Hilscher, R., Zeidan, V.:
Calculus of variations on time scales: weak local piecewise
$C\sp 1\sb {\rm rd}$ solutions with variable endpoints.
J. Math. Anal. Appl.
\textbf{289}, no.~1, 143--166 (2004)

\bibitem{MR2049870}
Hydon, P. E., Mansfield, E. L.:
A variational complex for difference equations.
Found. Comput. Math. \textbf{4}, no.~2, 187--217 (2004)

\bibitem{QC}
Kac, V., Cheung, P.:
Quantum calculus.
Universitext, Springer, New York (2002)

\bibitem{Lakshmikantham}
Lakshmikantham, V., Sivasundaram, S., Kaymakcalan, B.:
Dynamic systems on measure chains.
Kluwer Acad. Publ., Dordrecht (1996)

\bibitem{livro:Adv:CoV}
Malinowska, A. B., Odzijewicz, T., Torres, D. F. M.:
Advanced methods in the fractional calculus of variations.
Springer Briefs in Applied Sciences and Technology,
Springer, Cham (2015)

\bibitem{TorresDeltaNabla}
Malinowska, A. B., Torres, D. F. M.:
The delta-nabla calculus of variations.
Fasc. Math.
\textbf{44}, 75--83 (2010)
{\tt arXiv:0912.0494}

\bibitem{MalinowskaTorresCompositionDelta}
Malinowska, A. B., Torres, D. F. M.:
Euler-Lagrange equations for composition functionals
in calculus of variations on time scales.
Discrete Contin. Dyn. Syst.
\textbf{29}, no.~2, 577--593 (2011)
{\tt arXiv:1007.0584}

\bibitem{MalinowskaTorresCompositionNabla}
Malinowska, A. B., Torres, D. F. M.:
A general backwards calculus of variations via duality.
Optim. Lett. \textbf{5}, no.~4, 587--599 (2011)
{\tt arXiv:1007.1679}

\bibitem{livro:FC:Int}
Malinowska, A. B., Torres, D. F. M.:
Introduction to the fractional calculus of variations.
Imperial College Press, London (2012)

\bibitem{MyID:266}
Malinowska, A. B., Torres, D. F. M.:
Quantum variational calculus.
Springer Briefs in Electrical and Computer Engineering,
Springer, Cham (2014)

\bibitem{natalia:CV}
Martins, N., Torres, D. F. M.:
Calculus of variations on time scales with nabla derivatives.
Nonlinear Anal.
\textbf{71}, no.~12, e763--e773 (2009)
{\tt arXiv:0807.2596}

\bibitem{Generalizing_the_variational_theory}
Martins, N., Torres, D. F. M.:
Generalizing the variational theory on time scales to include the delta indefinite integral.
Comput. Math. Appl.
\textbf{61}, no.~9, 2424--2435 (2011)
{\tt arXiv:1102.3727}

\bibitem{MR2966852}
Martins, N., Torres, D. F. M.:
Higher-order infinite horizon variational problems in discrete quantum calculus.
Comput. Math. Appl.
\textbf{64}, no.~7, 2166--2175 (2012)
{\tt arXiv:1112.0787}

\bibitem{MR2994055}
Martins, N., Torres, D. F. M.:
Necessary optimality conditions for higher-order infinite horizon variational problems on time scales.
J. Optim. Theory Appl.
\textbf{155}, no.~2, 453--476 (2012)
{\tt arXiv:1204.3329}

\bibitem{MR2028477}
Merrell, E., Ruger, R., Severs, J.:
First order recurrence relations on isolated time scales.
Panamer. Math. J.
\textbf{14}, no.~1, 83--104 (2004)

\bibitem{MyID:288}
Odzijewicz, T., Torres, D. F. M.:
The generalized fractional calculus of variations.
Southeast Asian Bull. Math.
\textbf{38}, no.~1, 93--117 (2014)
{\tt arXiv:1401.7291}

\bibitem{orlov}
Orlov, I. V.:
Inverse extremal problem for variational functionals.
Eurasian Math. J.
\textbf{1}, no.~4, 95--115 (2010)

\bibitem{MR2907362}
Orlov, I. V.:
Elimination of Jacobi equation in extremal variational problems.
Methods Funct. Anal. Topology
\textbf{17}, no.~4, 341--349 (2011)

\bibitem{MR2732180}
Saunders, D. J.:
Thirty years of the inverse problem in the calculus of variations.
Rep. Math. Phys. \textbf{66}, no.~1, 43--53 (2010)

\bibitem{Rahmat}
Segi Rahmat, M. R.:
On some $(q,h)$-analogues of integral inequalities on discrete time scales.
Comput. Math. Appl.
\textbf{62}, no.~4, 1790--1797 (2011)

\bibitem{MR2098297}
Torres, D. F. M.:
Proper extensions of Noether's symmetry theorem for nonsmooth extremals of the calculus of variations.
Commun. Pure Appl. Anal.
\textbf{3}, no.~3, 491--500 (2004)

\bibitem{Torres}
Torres, D. F. M.:
The variational calculus on time scales.
Int. J. Simul. Multidisci. Des. Optim.
\textbf{4}, no.~1, 11--25 (2010)
{\tt arXiv:1106.3597}

\bibitem{MR2004181}
van Brunt, B.:
The calculus of variations.
Universitext, Springer, New York (2004)

\bibitem{Wyrwas}
Wyrwas, M.:
Introduction to control systems on time scales.
Lecture Notes, Institute of Cybernetics,
Tallinn University of Technology (2007)

\end{thebibliography}
\end{document}